\pdfoutput=1
\RequirePackage[l2tabu, orthodox]{nag}

\documentclass[reqno]{amsart}

\usepackage{lmodern}
\usepackage[T1]{fontenc}
\usepackage[utf8]{inputenc}
\usepackage[english]{babel}
\usepackage{microtype} %Better font spacing

\usepackage{amsmath,amssymb,amsthm,mathrsfs,latexsym,mathtools,mathdots,booktabs,enumerate,tikz,bm,url}
\usepackage[enableskew,vcentermath]{youngtab}
\usepackage[centertableaux]{ytableau}
\usepackage{enumitem} 

\usetikzlibrary{arrows,positioning,decorations.pathreplacing}
\usepackage[capitalize,noabbrev]{cleveref}

\usepackage{todonotes}
\presetkeys%
    {todonotes}%
    {inline,backgroundcolor=yellow}{}

\usepackage{graphicx}

\usepackage{ifpdf}
\graphicspath{{./figures/}{./}}
\usepackage{epstopdf}
\DeclareGraphicsExtensions{.png,.jpg,.eps,.epsf,.pdf}

\newtheorem{theorem}{Theorem}[section]
\newtheorem{proposition}[theorem]{Proposition}
\newtheorem{lemma}[theorem]{Lemma}
\newtheorem{corollary}[theorem]{Corollary}
\newtheorem{conjecture}[theorem]{Conjecture}

\theoremstyle{definition}
\newtheorem{definition}[theorem]{Definition}
\newtheorem{construction}[theorem]{Construction}
\newtheorem{example}[theorem]{Example}
\newtheorem{remark}[theorem]{Remark}

\numberwithin{equation}{section}

\newcommand{\defin}[1]{\emph{#1}}

\newcommand{\setN}{\mathbb{N}}
\newcommand{\setZ}{\mathbb{Z}}

\newcommand{\setR}{\mathbb{R}}
\newcommand{\setC}{\mathbb{C}}

\newcommand{\avec}{\mathbf{a}}
\newcommand{\bvec}{\mathbf{b}}

\newcommand{\rvec}{\mathbf{r}}

\newcommand{\xvec}{\mathbf{x}}

\newcommand{\zvec}{\mathbf{z}}

\newcommand{\uvec}{\mathbf{u}}
\newcommand{\vvec}{\mathbf{v}}
\newcommand{\svec}{\mathbf{s}}
\newcommand{\tvec}{\mathbf{t}}

\newcommand{\nthroot}{\omega}
\newcommand{\nthrootvec}{\boldsymbol{\omega}}

\newcommand{\orbit}{\mathcal{O}}

\newcommand{\qbinom}{\genfrac{[}{]}{0pt}{}}

\newcommand{\schurS}{\mathrm{s}}

\newcommand{\CSP}{\mathrm{CSP}}
\newcommand{\SSYT}{\mathrm{SSYT}}

\newcommand{\Mod}[1]{\ (\mathrm{mod}\ #1)}

\def\multiset#1#2{\ensuremath{\left(\kern-.3em\left(\genfrac{}{}{0pt}{}{#1}{#2}\right)\kern-.3em\right)}}

\DeclareMathOperator{\maj}{maj}

\DeclareMathOperator{\stat}{\tau}

\DeclareMathOperator{\flex}{flex}

\title{The cone of cyclic sieving phenomena}

\author{Per Alexandersson}
\author{Nima Amini}
\address{Dept. of Mathematics, Royal Institute of Technology, SE-100 44 Stockholm, Sweden}
\email{per.w.alexandersson@gmail.com, namini@kth.se}

\begin{document}
\begin{abstract}
We study cyclic sieving phenomena (CSP) on combinatorial objects from an abstract point of view by considering
a rational polyhedral cone determined by the linear equations that define such phenomena. 
Each lattice point in the cone corresponds to a non-negative integer matrix which jointly 
records the statistic and cyclic order distribution associated with the set of objects realizing the CSP.
In particular we consider a \emph{universal} subcone onto which every CSP matrix linearly 
projects such that the projection realizes a CSP with the same cyclic orbit structure, 
but via a \emph{universal} statistic that has even distribution on the orbits. 

Reiner et.al. showed that every cyclic action give rise to a unique polynomial (mod $q^n-1$) 
complementing the action to a CSP. 
We give a necessary and sufficient criterion for the converse to hold. 
This characterization allows one to determine if a combinatorial set with a 
statistic give rise (in principle) to a CSP without having a combinatorial 
realization of the cyclic action. We apply the criterion to conjecture a new CSP involving 
stretched Schur polynomials and prove our conjecture for certain rectangular tableaux. 
Finally we study some geometric properties of the CSP cone. 
We explicitly determine its half-space description and in the prime order case we determine its extreme rays.
\end{abstract}
\maketitle

\tableofcontents

\section{Introduction}

\subsection{Background on cyclic sieving phenomena}

The cyclic sieving phenomenon was introduced by Reiner, Stanton and White in \cite{Reiner2004}.
For a survey, see \cite{Sagan2011}.
\begin{definition}
Let $C_n$ be a cyclic group of order $n$ generated by $\sigma_n$, $X$ a finite set on which $C_n$ acts and $f(q) \in \mathbb{N}[q]$.
Let $X^{g} \coloneqq \{ x \in X : g \cdot x = x\}$ denote the fixed point set of $X$ under $g \in C_n$.
We say that the triple $(X,C_n,f(q))$ exhibits the \emph{cyclic sieving phenomenon (CSP)} if
\begin{equation} \label{eq:origCSPCond}
f(\nthroot_n^k) = |X^{\sigma_n^k}|, \text{ for all $k \in \mathbb{Z}$},
\end{equation}
where $\nthroot_n$ is any fixed primitive $n^{\text{th}}$ root of unity.
\end{definition}

Since $f(1)$ is always the cardinality of $X$, it is common that
$f(q)$ is given as $f_{\stat}(q) \coloneqq \sum_{x\in X} q^{\stat(x)}$ for some statistic on $X$.
With this in mind, we say that the triple $(X,C_n,\stat)$ exhibits CSP if $(X,C_n,f_{\stat}(q))$ does.

\medskip 
Here is a short list of cyclic sieving phenomena found in the literature (see \cite{Reiner2004, Sagan2011} for a more comprehensive list):
\begin{itemize}
 \item Words $X = W_{n,k}$ of length $n$ over an alphabet of size $k$, $C_n$ acting via cyclic shift, 
 \[ f(q) \coloneqq \qbinom{n+k-1}{k}_q  = \sum_{w\in W_{n,k}} q^{\maj w}.
 \]

 \item Standard Young tableaux $X = \text{SYT}(\lambda)$ of rectangular shape $\lambda = (n^m)$, $C_n$ acting via jeu-de-taquin promotion \cite{Rhoades2010}, 
\[
f(q) \coloneqq \frac{[n]_q!}{\prod_{(i,j) \in \lambda} [h_{i,j}]_q} = q^{-n \binom{m}{2}}\sum_{T \in \mathrm{SYT}(\lambda)} q^{\maj(T)},
\]
this expression being the $q$-hook-length formula \cite{Stanley1971}.

\item Triangulations $X$ of a regular $(n+2)$-gon, $C_{n+2}$ acting via rotation
of the triangulation, $f(q) \coloneqq \frac{1}{[n+1]_q} \qbinom{2n}{n}_q$, MacMahon's $q$-analogue of the Catalan numbers \cite{MacMahon1960}.
Note that through well-known bijections (see \cite{StanleyCatalan2015}) we get induced CSPs with the sets $X = \mathrm{Dyck}(n)$, the set of Dyck paths of semi-length $n$,
and $X = \mathfrak{S}_n(231)$, the set of permutations in $\mathfrak{S}_n$ avoiding the classical pattern $231$.
Moreover one has
\[
f(q) \coloneqq \frac{1}{[n+1]_q} \qbinom{2n}{n}_q = \sum_{P \in \mathrm{Dyck}(n)} q^{\maj(P)} = \sum_{\pi \in \mathfrak{S}_n(231)} q^{\maj(\pi) + \maj(\pi^{-1})},
\]
where the last equality is due to Stump \cite{Stump2009}.
\end{itemize} \noindent

\subsection{Outline of the paper}
The examples presented in the previous subsection have one or more of the following pair of common features:
\begin{itemize}
\item The action of $C_n$ on $X$ has a \emph{natural} definition.
\item The polynomial $f(q)$ is generated by a \emph{natural} statistic on $X$.
\end{itemize}\noindent
What is \emph{natural} largely lies in the eyes of the beholder, but broadly it could be taken to mean a definition with combinatorial substance.

The following equivalent condition for a triple $(X,C_n,f(q))$
to exhibit the cyclic sieving phenomenon was given by Reiner--Stanton--White in \cite{Reiner2004}:
\begin{equation} \label{eq:equivDefRSW}
f(q) \equiv \sum_{\mathcal{O} \in \text{Orb}_{C_n}(X)} \frac{q^{n}-1}{q^{n/|\mathcal{O}|}- 1} \Mod{q^n-1},
\end{equation}
where $\text{Orb}_{C_n}(X)$ denotes the set
of orbits of $X$ under the action of $C_n$. 

Therefore the coefficient of $q^{i}$ in $f(q) \Mod{q^n-1}$ is generically interpreted as the number of orbits whose stabilizer-order divides $i$. This alternative condition also means that every cyclic action of $C_n$ on a finite set $X$ give rise to a (not necessarily natural) polynomial $f(q)$, unique modulo $q^n-1$, such that $(X,C_n,f(q))$ exhibits the cyclic sieving phenomenon.

In this paper we consider when the converse of the above property holds. \emph{Given a combinatorial set $X$ with a natural statistic $\tau:X \to \mathbb{N}$, when does it give rise to a (not necessarily natural) action of $C_n$ on $X$ such that $(X,C_n,\tau)$ exhibits the cyclic sieving phenomenon?}

Having a necessary and sufficient criteria for the existence of such a CSP adds a couple of benefits:  

\begin{itemize}
\item Given a polynomial $f(q) = \sum_{x \in X} q^{\tau(x)}$ generated by a natural statistic $\tau:X \to \mathbb{N}$, we can determine if a CSP exists in principle without knowing a combinatorial realization of the cyclic action. The criteria thus serves as a tool for confirming or refuting the existence of cyclic sieving phenomena involving a candidate polynomial.
\item Generic evidence that a CSP exists provides motivation to search for a combinatorially meaningful cyclic action on the set $X$.  
\end{itemize}
\medskip 
The main result in \cref{sec:intValPoly} is the following:
\cref{thm:cspCond} provides the necessary and sufficient conditions for $(X,C_n,f(q))$ to exhibit CSP.
The natural (necessary) condition is that $f(q) \in \mathbb{Z}[q]$ take non-negative integer values at all $n^{\text{th}}$ roots of unity,
which is evident from the definition of a cyclic sieving phenomena.

We prove the following: Define
\[
S_k \coloneqq \sum_{j|k} \mu(k/j) f(\omega_n^j), \qquad \text{ where $k|n$}.
\] 
Then $(X,C_n,f(q))$ exhibits CSP if and only if $S_k \geq 0$ for all $k|n$.
% The integer $S_k$ determine the number of elements of order $k$.
\medskip 
We warn that merely having a polynomial $f(q) \in \mathbb{N}[q]$ that takes non-negative integer values at all $n^{\text{th}}$ roots of unity is no guarantee for the existence of a cyclic action complementing $f(q)$ to a CSP. A polynomial demonstrating this is given in \cref{ex:counterex}.

In \cref{sec:apps}, we conjecture a new cyclic sieving phenomena
involving stretched Schur polynomials.
In a special case, we prove this conjecture by applying \cref{thm:cspCond}, see \cref{thm:rectSchur} below.
That is, we prove existence of CSP without having to provide a natural cyclic group action.

\cref{sec:cspCone} and onwards treat the cyclic sieving phenomenon 
from a more geometric perspective. 
We record the joint cyclic order and statistic distribution of the elements of $X$ 
in a matrix and reformulate the CSP condition in terms of linear equations in the matrix entries. 
The set of matrices that satisfy these linear equations we call \emph{CSP matrices} 
and we prove via \cref{thm:nhypdesc} that they form a convex 
rational polyhedral cone whose integer lattice points correspond to 
realizable instances of CSP. 
Inspired by \cite{Ahlbach2017}, we further proceed to identify a certain 
subcone which we call the \emph{universal CSP cone} containing all matrices corresponding to realizable 
instances of CSP with evenly distributed statistic on all its orbits. 
We prove that all integer CSP matrices can be obtained from a universal CSP matrix 
through a sequence of \emph{swaps} 
without going outside of the CSP cone (\cref{prop:swap}). 
The swaps can be interpreted as a sequence of statistic 
interchanges between pairs of elements in the corresponding CSP-instance. 

Finally we explicitly determine all extreme rays of the universal CSP cone 
(\cref{cor:univExtremeRays}) and in \cref{sec:generalProps} 
we prove some general properties for all CSP cones.

\subsection{Notation}
The following notation will be used throughout the paper.
\begin{itemize}
\item $[n] \coloneqq \{1,\dots, n \}$.
\item $\mathbb{R}_{\geq 0}$ denotes the set of non-negative real numbers.
\item $K^{n \times n}$ denotes the set of $n \times n$ matrices over the set $K$.
%\item $||A|| = \sum_{i,j} |a_{ij}|$ where $A = (a_{ij})$.
\item $\mu(n) \coloneqq \begin{cases} 0, & \text{ if $n$ is not square-free}, \\ (-1)^r, & \text{ if $n$ is a product of $r$ distinct primes}, \end{cases}$ \newline
denotes the classical M\"obius function.
\item $\omega_n$ denotes a primitive $n^{\text{th}}$ root of unity. 
\item $\displaystyle \Phi_n(q) \coloneqq \prod_{\substack{1 \leq k \leq n \\ \text{gcd}(n,k) = 1}} (q- \omega_n^k)$ denotes the $n^{\text{th}}$ cyclotomic polynomial.  
\item 
$\displaystyle [n]_q \coloneqq \frac{q^n-1}{q-1}$, \hspace{0.1cm} $\displaystyle [n]_q! \coloneqq [n]_q[n-1]_q \cdots [1]_q$, \hspace{0.1cm} $\displaystyle \qbinom{n}{k}_q \coloneqq \frac{[n]_q!}{[k]_q![n-k]_q!}$, \newline 
denotes the $q$-integer, $q$-factorial and $q$-binomial coefficients respectively.
\end{itemize}

\section{Integer-valued polynomials at roots of unity} \label{sec:intValPoly}
In the context of discovering cyclic sieving phenomena, one may sometimes have a
candidate polynomial (\emph{e.g.} a natural $q$-analogue of the enumeration formula for the underlying set) that takes integer values at all roots of unity, but the cyclic action complementing it to a CSP is unknown. In such situations one may like to know if a CSP could exist even in principle. In this section we characterize the set polynomials $f(q) \in \mathbb{Z}[x]$ of degree less than $n$ such that $f(\nthroot_n^j) \in \mathbb{Z}$ for all $j = 1, \dots, n$ and show that they are indeed $\mathbb{Z}$-linear combinations of polynomials of the form
\[
\frac{q^n-1}{q^{n/d}-1} = \sum_{i=0}^{n/d -1} q^{di} \hspace{0.5cm} \text{ for } d|n.
\]
Using the characterization one can quickly determine if a CSP is present and get the count of the number of elements of each order in terms of evaluations of the polynomial at roots of unity. Often it is much simpler to determine the evaluations at roots of unity than it is to write the polynomial in terms of the above basis.

Finally note that not all polynomials $f(q) \in \mathbb{N}[q]$ such that $f(\nthroot_n^j) \in \mathbb{N}$ for all $j = 1, \dots, n$ may necessarily be paired with a cyclic action to produce a CSP, see Example \ref{ex:counterex}.

The set 
\[
M(n) \coloneqq \{ f(q) \in \mathbb{Z}[q] : \text{deg}(f) < n, \thickspace f(\nthroot_n^j)\in \mathbb{Z} \text{ for } j= 1, \dots, n \}
\]
forms a $\mathbb{Z}$-module. First we identify two useful bases for $M(n)$ using the following proposition and \cref{lem:modBasis}.
\begin{proposition}[Désarménien \cite{Desarmemien1989}]\label{prop:Desarmien}
Let $f(q)\in \mathbb{Z}[q]$ be a polynomial of degree less than $n$.
Then the following two properties are equivalent:
\begin{enumerate}[label=(\roman*)]
\item For every $d|n$,
\[
f(q) \equiv r_d \Mod{\Phi_d(q)} \hspace{0.2cm} \text{for some } r_d \in \mathbb{Z},
\]
where $\Phi_d(q)$ denotes the $d^{\text{th}}$ cyclotomic polynomial.
\item The polynomial $f(q)$ has the form
\begin{align} \label{eq:basisForm}
f(q) = \sum_{j=0}^{n-1} a_j q^j, \hspace*{0.2cm} \text{ where } a_j = a_{\text{gcd}(n,j)}.
\end{align}
\end{enumerate}
\end{proposition} \noindent
\begin{lemma} \label{lem:modBasis}
For each $n \in \mathbb{N}$, the following sets form $\mathbb{Z}$-bases for $M(n)$:
\begin{enumerate}[label=(\roman*)]
\item $\mathcal{B}_1(n) = \{ g_d(q) : d|n \}$ where
\[
g_d(q) = \sum_{\substack{0 \leq j < n \\ \gcd(j,n) = d}} q^j,
\]
\item $\mathcal{B}_2(n) = \{ h_d(q) : d|n \}$ where
\[
h_d(q) = \sum_{j=0}^{n/d-1} q^{dj}.
\]
\end{enumerate}
\end{lemma}
\begin{proof}
Let $f(q) \in M(n)$ and suppose $d|n$.
Then $\nthroot_n^{n/d}$ is a $d^{\text{th}}$ root of unity.
Note that $f(q) - f(\nthroot_n^{n/d})$ vanishes at $q = \nthroot_n^{n/d}$ so it is divisible by the minimal polynomial of $\nthroot_n^{n/d}$ over $\mathbb{Z}$, that is, $\Phi_d(q)$.
Hence
$f(q) \equiv r_d \Mod{\Phi_d(q)}$ where $r_d= f(\nthroot_n^{n/d}) \in \mathbb{Z}$.
By \cref{prop:Desarmien} it follows that $f(q)$ has the form \eqref{eq:basisForm}.
Such polynomials are clearly spanned by $\mathcal{B}_1(n)$.

Now, the elements in $\mathcal{B}_2(n)$ are linearly independent, since the lowest-degree
terms of $h_d(q)-1$ are all different.
By inclusion-exclusion we see that for each $d|n$,
\[
g_d(q) = \sum_{d|r} \mu(r/d) h_r(q)
\]
and hence $\mathcal{B}_1(n)$ and $\mathcal{B}_2(n)$ both form bases of $M(n)$.
\end{proof}
We may in fact extend the characterization in \cref{lem:modBasis} to multivariate polynomials $f \in \mathbb{Z}[q_1,\dots, q_m]$ of degree less than $n_i$ in variable $q_i$ for $i = 1,\dots, m$ taking integer values at all points $(\omega_{n_1}^{j_1}, \dots, \omega_{n_m}^{j_m}) \in \mathbb{C}^m$ for $j_i = 1,\dots, n_i$, $i = 1,\dots, m$.
\begin{theorem}\label{thm:multmodbasis}
Let 
$M(n_1,\dots, n_m) = \{ f\in \mathbb{Z}[q_1,\dots, q_m] :\text{deg}_i f < n_i, \thickspace f(\omega_{n_1}^{j_1},\dots, \omega_{n_m}^{j_m}) \in \mathbb{Z} \text{ for } j_i = 1, \dots, n_i, \thickspace i = 1, \dots, m\}$ where $n_1,\dots, n_m \in \mathbb{N}$ and $\text{deg}_i f$ denotes the degree of $x_i$ in $f$. Then the following sets form $\mathbb{Z}$-bases for $M(n_1,\dots, n_m)$:
\begin{enumerate}[label=(\roman*)]
\item $\mathcal{B}_1(n_1,\dots, n_m) = \left \{ \prod_{i=1}^m g_{d_i}^{(i)}(q_i) : d_i|n_i, \thickspace i = 1,\dots, m \right \}$ where
\[
g_{d_i}^{(i)}(q_i) = \sum_{\substack{0 \leq j < n_i \\ \text{gcd}(j,n_i) = d_i}} q_i^j,
\]
\item $\mathcal{B}_2(n_1,\dots, n_m) = \left \{ \prod_{i=1}^m h_{d_i}^{(i)}(q_i) : d_i|n_i, \thickspace i = 1,\dots, m \right \}$ where
\[
h_{d_i}^{(i)}(q_i) = \sum_{j=0}^{n_i/d_i -1} q_i^{d_ij}.
\]
\end{enumerate}
\end{theorem}
\begin{proof}
We prove that $\mathcal{B}_1(n_1,\dots, n_m)$ is a $\mathbb{Z}$-basis of $M(n_1,\dots, n_m)$ by induction on $m$. The proof for $\mathcal{B}_2$ is similar and therefore omitted. The base case $m=1$ follows from Lemma \ref{lem:modBasis}. Let $f \in M(n_1,\dots, n_{m+1})$. Write 
\[
f = f_{n_{m+1}-1}(q_1,\dots,q_{m})q_{m+1}^{n_{m+1}-1} + \cdots + f_{1}(q_1,\dots,q_{m})q_{m+1} + f_{0}(q_1,\dots,q_{m}),
\] 
where $f_{0},f_{1}, \dots, f_{n_{m+1}-1} \in \mathbb{Z}[q_1,\dots, q_{m}]$ with $f_k(\omega_{n_1}^{j_1}, \dots, \omega_{n_{m}}^{j_{m}}) \in \mathbb{Z}$ for all $k = 0,\dots, n_{m+1}-1$, $j_i = 1, \dots, n_i$ and $i = 1,\dots, m$. The univariate polynomials \[
F_{\omega_{n_1}^{j_1}, \dots, \omega_{n_{m}}^{j_{m}}}(q_{m+1}) = f(\omega_{n_1}^{j_1}, \dots, \omega_{n_{m}}^{j_{m}},q_{m+1}) \in \mathbb{Z}[q_{m+1}],
\] 
take integer values at $q_{m+1} = \omega_{n_{m+1}}^{j}$ for all $j = 1,\dots, n_{m+1}$. By Proposition \ref{prop:Desarmien} we therefore have that
\[
f_k(\omega_{n_1}^{j_1}, \dots, \omega_{n_{m}}^{j_{m}}) = f_{\text{gcd}(n_{m+1},k)}(\omega_{n_1}^{j_1}, \dots, \omega_{n_{m}}^{j_{m}}),
\]
for all $(\omega_{n_1}^{j_1}, \dots, \omega_{n_m}^{j_m}) \in \mathbb{C}^m$. Since the $\prod_{i=1}^m n_i$ points $(\omega_{n_1}^{j_1}, \dots, \omega_{n_m}^{j_m}) \in \mathbb{C}^m$ lie in general position the polynomials must coincide on all points in $\mathbb{C}^m$. Hence 
\[ 
f_k(q_1,\dots, q_m) = f_{\text{gcd}(n_{m+1},k)}(q_1,\dots, q_m)
\]
for all $k = 0,\dots, n_{m+1}-1$. It follows that $f$ is uniquely spanned by $\mathcal{B}_1(n_{m+1})$ over $\mathbb{Z}[q_1,\dots, q_m]$. By induction $f_k(q_1,\dots, q_m)$ is uniquely spanned by $\mathcal{B}_1(n_1,\dots, n_m)$ over $\mathbb{Z}$ for all $k = 0,\dots, n_{m+1}-1$. Hence $f$ is uniquely spanned by $\mathcal{B}_1(n_1,\dots, n_{m+1})$ over $\mathbb{Z}$ completing the induction.
\end{proof}

\begin{lemma} \label{lem:peCong}
Let $f(q) \in \mathbb{Z}[q]$ such that $f(\nthroot_n^j) \in \mathbb{Z}$ for all $j = 1, \dots, n$. Then for each $m,p,e \in \mathbb{N}$ where $p$ is prime we have
\[
f(\nthroot_n^{mp^e}) \equiv f(\nthroot_n^{mp^{e-1}}) \Mod{p^e}.
\]
In particular if $p\not | n$, then $f(\nthroot_n^{mp^{e-1}}) =  f(\nthroot_n^{mp^e})$.
\end{lemma}
\begin{proof}
Since we are only concerned with evaluations of $f(q)$ at $n^{\text{th}}$ roots of unity, we may assume $f(q)\in M(n)$. 
Furthermore by \cref{lem:modBasis} and linearity it suffices to show the 
statement for the basis elements $\mathcal{B}_2$ of $M(n)$. 
For each $d|n$ and $k \in \mathbb{Z}$ we have
\[
h_d(\nthroot_n^{k}) = \sum_{j=0}^{n/d - 1} (\nthroot_{n/d}^{k})^j = 
\begin{cases} n/d, & \text{ if } k \equiv 0 \Mod{n/d}, \\ 
0, & \text{ otherwise}. 
\end{cases}
\]
Now suppose $k = mp^e$ for some $m,p,e \in \mathbb{N}$ with $p$ prime, and consider the different cases: Suppose first $mp^{e-1} \equiv 0 \Mod{n/d}$. 
This implies that $mp^{e} \equiv 0 \Mod{n/d}$, so $h_d(\nthroot_n^{mp^e}) = n/d = h_d(\nthroot_n^{mp^{e-1}})$.
Secondly, suppose $mp^{e-1} \not \equiv 0 \Mod{n/d}$. 
If $mp^{e} \not \equiv 0 \Mod{n/d}$, then $h_d(\nthroot_n^{mp^e}) = 0 = h_d(\nthroot_n^{mp^{e-1}})$. 
On the other hand if $mp^{e} \equiv 0 \Mod{n/d}$, then $n/d = p^fa$ for some $f \geq e$ and $a \in \mathbb{N}$. 
Therefore $h_d(\nthroot_n^{mp^e}) - h_d(\nthroot_n^{mp^{e-1}}) = p^fa - 0 \equiv 0 \Mod{p^e}$. 
Hence the lemma follows.
\end{proof} \noindent
\begin{lemma} \label{lem:mobiusDiv}
Let $f(q) \in \mathbb{Z}[q]$ such that $f(\nthroot_n^j) \in \mathbb{Z}$ for all $j = 1, \dots, n$. Then for each $k = 1,\dots, n$ we have that
\[
\sum_{j|k} \mu(k/j)f(\nthroot_n^j) \equiv 0 \Mod{k}.
\]
Moreover if $k \not |n$, then $\sum_{j|k} \mu(k/j)f(\nthroot_n^j) = 0$.
\end{lemma}
\begin{proof}
Let $1 \leq k \leq n$ and write $k = mp^e$ where $p,m\in \mathbb{N}$, $p$ prime and $p \not | m$.
By Lemma \ref{lem:peCong} we have
\begin{align*}
\sum_{j|k} \mu(k/j)f(\nthroot_n^j) &= \sum_{j|m} \mu(k/(jp^{e-1}))f(\nthroot_n^{jp^{e-1}}) + \sum_{j|m} \mu(k/(jp^{e}))f(\nthroot_n^{jp^{e}}) \\ &\equiv \sum_{j|m} \mu(k/(jp^{e-1}))f(\nthroot_n^{jp^{e-1}}) + \sum_{j|m} \mu(k/(jp^{e}))f(\nthroot_n^{jp^{e-1}}) \Mod{p^e} \\ &\equiv 0 \Mod{p^e}.
\end{align*} \noindent
If $k \not |n$, then we may write $k= mp^e$ for some $m,p \in \mathbb{N}$ with $p$ prime such that $p \not |n$. Then by the second assertion in Lemma \ref{lem:peCong} the congruences above hold with equality and we are done.
\end{proof}
\begin{construction} \label{con:adhoc}
Let $X = \mathcal{O}_1 \sqcup \mathcal{O}_2 \sqcup \cdots \sqcup \mathcal{O}_m$ be a partition of a finite set $X$ into $m$ parts such that $|\mathcal{O}_i|$ divides $n$ for $i = 1, \dots, m$. Fix a total ordering on the elements of $\mathcal{O}_i$ for $i = 1, \dots, m$. Let $C_n$ act on $X$ by permuting each element $x \in \mathcal{O}_i$ cyclically with respect to the total ordering on $\mathcal{O}_i$ for $i = 1,\dots, m$. 
\end{construction} \noindent
This \textit{ad-hoc} cyclic action in \cref{con:adhoc} lacks combinatorial context and depends only on the choice of partition and total order.
\begin{theorem} \label{thm:cspCond}
Let $f(q)\in \mathbb{N}[q]$ and suppose $f(\nthroot_n^j) \in \mathbb{N}$ 
for each $j=1, \dots, n$. Let $X$ be any set of size $f(1)$. Then there exists an action of $C_n$ on $X$ such that
$(X,C_n,f(q))$ exhibits CSP if and only if for each $k|n$,
\begin{align} \label{poscond}
\sum_{j | k} \mu(k/j)f(\nthroot_n^j) \geq 0.
\end{align}
\end{theorem}
\begin{proof}
The forward direction follows from \cite[Prop. 4.1]{Reiner2004}. 
%Indeed if there exists a cyclic action of $C_n$ on a set $X$ such that $(W,C_n,f(q))$ is a CSP-triple, then letting 
%$S_j = |\{ x \in X: o(x) = j \}|$ for $j = 1,\dots, n$ we have
%\[
%f(\nthroot_n^k) = |X^{\sigma^k}| =  \sum_{j|k} S_j,
%\]
%for each $k = 1, \dots, n$. By Möbius inversion we get
%\[
%\sum_{j|k} \mu(k/j)f(\nthroot_n^j) = S_k \geq 0,
%\]
%for $k = 1, \dots, n$ as required. 
Conversely if we put 
\begin{align}\label{eq:columnSums}
S_k = \sum_{j|k} \mu(k/j)f(\nthroot_n^j)
\end{align}
for each $k= 1, \dots, n$ and consider $X$ of size $f(1)$, then by Möbius inversion
\[
|X| = f(\nthroot_n^n) = \sum_{j|n} S_j.
\]
Thus by hypothesis and Lemma \ref{lem:mobiusDiv}, we may partition $X$ into orbits, such that for each $k|n$, there are $\frac{1}{k}S_k$ orbits of size $k$. We then let $C_n$ act on $X$ as in \cref{con:adhoc}. The fixed points of $X$ under $\sigma_n^k \in C_n$ are given by the elements of order dividing $k$. This gives (by Möbius inversion)
\[
|X^{\sigma_n^k}| = \sum_{j|k} S_j = f(\nthroot_n^{k}).
\]
Hence $(X,C_n,f(q))$ exhibits CSP.
\end{proof} \noindent
\begin{remark} \label{rem:numElements}
The sums $S_k$ in \eqref{eq:columnSums} represent the number of elements with order $k$ under the action of $C_n$.
\end{remark} \noindent
\begin{example} \label{ex:counterex}
The following example demonstrates that even if $f(q)\in \mathbb{N}[q]$ 
satisfies $f(\nthroot_n^j) \in \mathbb{N}$ for all $j = 1,\dots, n$, 
there might not be an associated cyclic action complementing $f(q)$ to a CSP.

Let $f(q) = q^5 + 3q^3 + q + 10$. Then $f(\nthroot_6^j)$ takes 
values $8,12,5,12,8,15$ for $j = 1,\dots, 6$. 
On the other hand $S_k = \sum_{j|k} \mu(k/j) f(\nthroot_6^j)$ takes 
values $8,4,-3,0,0,6$ for $k = 1, \dots, 6$. 
Since we cannot have a negative number of elements of order $3$, 
there is no action of $C_6$ on a set $X$ of size $f(1) = 15$ such that $(X,C_6,f(q))$ is a CSP-triple.
\end{example} \noindent
Rao and Suk \cite{Rao2017} generalized the notion of cyclic sieving to arbitrary groups with finitely generated representation ring, 
so called \emph{$G$-sieving}. 
In particular, Berget, Eu and Reiner \cite{Berget2011} considered the case where $G$ is 
an Abelian group, whence $G  \cong C_{n_1} \times \cdots \times C_{n_m}$, acting pointwise on a set $X_1 \times \cdots \times X_m$. 
Unfortunately $G$-sieving depends in general on the particular choices of representations $\rho_i$ of $G$ over $\mathbb{C}$ generating the 
representation ring. However, given the characterization in \cref{thm:multmodbasis} it would be interesting to 
understand what conditions are necessary and sufficient for a polynomial $f \in M(n_1,\dots, n_m)$ to be complemented 
to a $G$-sieving phenomenon for an Abelian group $G \cong C_{n_1} \times \cdots \times C_{n_m}$ with respect to 
the canonical representations sending the generator $\sigma_{n_i}$ of $C_{n_i}$ to $\omega_{n_i}$.

\section{Applications} \label{sec:apps}
In this section we demonstrate how one can use \cref{thm:cspCond} to
find new cyclic sieving phenomena arising from natural polynomials.

By \cref{thm:cspCond} any polynomial $f(q) \in \mathbb{N}[q]$ such that $f(\omega_n^j) \in \mathbb{N}$
for $j= 1,\dots, n$ satisfying the positivity condition (\ref{poscond}), can be
completed to a CSP with an ad-hoc cyclic action. Although this action lacks combinatorial context,
it often helps to know that a CSP can exist even in principle,
particularly if one is considering a combinatorial set where the cyclic action is not immediately apparent.
The following example illustrates this point for the polynomial $C_n(q) \coloneqq \frac{1}{[n+1]_q} \qbinom{2n}{n}_q$ which is generated by statistics on multiple combinatorial (Catalan) objects,
but where the naturalness of the action varies depending on the object under consideration.
\begin{example}
Stump \cite{Stump2009} showed that $C_n(q) = \sum_{\sigma \in \mathfrak{S}_n(231)} q^{\maj(\sigma) + \maj(\sigma^{-1})}$.
There is no obvious natural cyclic action on $\mathfrak{S}_n(231)$ that is compatible with $C_n(q)$.
However we can check the positivity condition \eqref{poscond} in \cref{thm:cspCond} to reveal that a CSP is nevertheless present for $C_n(q)$ with an ad-hoc cyclic action on $\mathfrak{S}_n(231)$.
Indeed rewriting $C_n(q) = \frac{1}{[2n+1]_q} \qbinom{2n+1}{n+1}_q$ and using \cite[Prop. 4.2 (iii)]{Reiner2004} we have for $j|n$,
\[
C_n(\omega_n^j) = \begin{cases}  \binom{2j}{j}, & \text{if } j < n, \\  \frac{1}{n+1}\binom{2n}{n}, & \text{if } j = n. \end{cases}
\]
By Wallis formula, $\prod_{n=1}^{\infty} \left ( 1 - \frac{1}{4n^2} \right ) = \frac{2}{\pi}$, the sequences 
\begin{align*}
2n \left ( \binom{2n}{n} \frac{1}{4^n} \right )^2 &= \frac{1}{2} \prod_{j=2}^n \left (1 + \frac{1}{4j(j-1)} \right ), \\ (2n+1) \left ( \binom{2n}{n} \frac{1}{4^n} \right )^2 &= \prod_{j=1}^n \left ( 1 - \frac{1}{4j^2} \right )  
\end{align*} \noindent
monotonically increase and decrease respectively towards $\frac{2}{\pi}$ as $n \to \infty$. Thus
\[
\frac{4^n}{\sqrt{\pi(n+ 1/2)}} \leq \binom{2n}{n} \leq \frac{4^n}{\sqrt{\pi n}}.
\]
A trivial bound for the number of divisors of $n$, excluding $n$, is given by $2\sqrt{n}-1$.
Hence for each divisor $k<n$ we have
\begin{align*}
\sum_{j|k} \mu(k/j) C_n(\omega_n^j) &= \sum_{j|k} \mu(k/j) \binom{2j}{j} \\ &\geq \binom{2k}{k} - \sum_{\substack{j|k \\ j < k}} \binom{2j}{j} \\ &\geq \frac{4^k}{\sqrt{\pi(k+1/2)}} - (2\sqrt{k}-1) \frac{4^{k/2}}{\sqrt{\pi(k/2)}} \geq 0.
\end{align*} \noindent
Moreover for $k = n$ we have by a similar calculation that
\[
\sum_{j|n} \mu(n/j) C_n(\omega_n^j) \geq \frac{4^n}{(n+1)\sqrt{\pi(n+1/2)}} - (2\sqrt{n}-1) \frac{4^{n/2}}{\sqrt{\pi(n/2)}} \geq 0,
\]
for $n \geq 5$. The required inequality can be verified explicitly by hand for $n < 5$.
Hence $C_n(q)$ exhibits CSP with an ad-hoc cyclic action on $\mathfrak{S}_n(231)$. 

With this evidence one could now either proceed to search for a natural cyclic
action on $\mathfrak{S}_n(231)$ matching the orbit structure of the ad-hoc
cyclic action, or find a natural cyclic action on an object in bijection
with $\mathfrak{S}_n(231)$. In this case there happens to exist known candidates \emph{e.g.}
the set of Dyck paths $\text{Dyck}(n)$ of semi-length $n$ where $C_n$ acts by
changing peaks to valleys (and vice versa) from left to right whenever possible,
or the set of triangulation of a regular $(n+2)$-gon where $C_{n+2}$ acts by
rotating the triangulation.
In the latter case we instead lack a simple natural
statistic (as opposed to a natural action) on the set of triangulations that generates $C_n(q)$.
\end{example} \noindent

\subsection{A new CSP with stretched Schur polynomials}

In this section we conjecture a new cyclic sieving phenomenon involving stretched Schur polynomials. 
We prove our conjecture in the case of certain rectangular shapes for which
it is straightforward to explicitly compute the data needed to verify 
the positivity condition \eqref{poscond} in \cref{thm:cspCond}. 
We begin by recalling the basic definitions required to state the conjecture.  

A \emph{partition} $\lambda = (\lambda_1, \dots, \lambda_r)$ is a finite weakly 
decreasing sequence of non-negative integers $\lambda_1 \geq \lambda_2 \geq \cdots \geq \lambda_r  \geq  0$. 
The \emph{parts} of $\lambda$ are the positive entries and the number of positive parts is the \emph{length} of $\lambda$, denoted $l(\lambda)$. The quantity $|\lambda| \coloneqq \lambda_1 + \dots + \lambda_r$ is called the \emph{size} of $\lambda$. The \emph{empty partition} $\emptyset$ is the partition with no parts.
We use exponents to denote multiplicities \emph{e.g.} $\lambda = (5,3,3,2,1,1,1) = (5,3^2,2,1^3)$. 
Scalar multiplication on partitions is performed elementwise \emph{e.g.} with $n \in \mathbb{N}$ and $\lambda$ as above we have $n\lambda = (5n,(3n)^2,2n,n^3)$. If $\mu = (\mu_1,\dots, \mu_r)$ is a partition such that $\lambda_i \geq \mu_i$ for all $i = 1,\dots, r$ then we say that $\mu \subseteq \lambda$. This is called the \emph{inclusion order} on partitions.

Partitions are commonly visualized in at least two different ways.
The first and most common way to represent a partition is via its Young diagram.
A \emph{skew Young diagram} of \emph{shape} $\lambda/\mu$ is an arrangement of boxes in 
the plane with coordinates given by $\{ (i,j) \in \mathbb{Z}^2 : \mu_i \leq j \leq \lambda_i \}$. 
The first coordinate represents the row and the second coordinate the column. 
If $\mu = \emptyset$, then we simply write $\lambda$ instead of $\lambda/\mu$ and refer 
to the corresponding skew Young diagram as the \emph{(regular) Young diagram} of $\lambda$. 
A \emph{border strip} (or \emph{rim hook}) of size $d$ is a connected skew Young 
diagram consisting of $d$ boxes and containing no $2 \times 2$ square. The \emph{height} of a border strip is one less than its number of rows. A \emph{border strip tableau} of shape $\lambda/\mu$ and type $\alpha = (\alpha_1, \dots, \alpha_d)$ is a sequence $\mu = \lambda^1 \subset \lambda^2 \subset \cdots \subset \lambda^r = \lambda$ such that $\lambda^i/\lambda^{i-1}$ is a border strip of size $\alpha_i$.

A second way to visually represent a partition $\lambda$ is via an \emph{abacus} with $m \geq r$ beads: 
Let $d \in \mathbb{N}$. For $i = 1,\dots, m$, write $\lambda_i + m - i = s+dt$, with $0 \leq s \leq d-1$, 
and place a bead on the $s^{th}$ runner in the $t^{th}$ row. 
The operation of sliding a bead one row upwards on its runner into a vacant 
position corresponds to removing a border strip of size $d$ from $\lambda$. 
Sliding all beads up as far as possible produces an abacus representation of the \emph{$d$-core} 
partition of $\lambda$, a partition from which no further border strip tableaux of size $d$ can be removed. 
It is worth mentioning that the $d$-core of $\lambda$ is independent of the way in which border strip tableaux are removed.
For $i=0,1,\dots, d-1$, let $\lambda_j^{(i)}$ be the number of unoccupied
positions on the $i^{th}$ runner above the $j^{th}$ bead from the bottom.
Then $\lambda^i = (\lambda_1^{(i)},\lambda_2^{(i)},\dots, \lambda_{d}^{(i)})$ is a partition and the $d$-tuple $[\lambda^{(0)},
\lambda^{(1)},\dots, \lambda^{(d-1)}]$ is called the \emph{$d$-quotient} of $\lambda$.
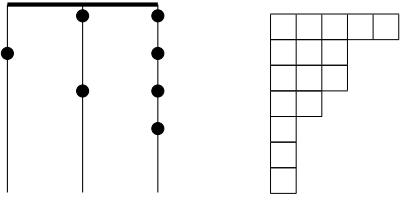
\begin{figure}
\[
\begin{tikzpicture}
\draw [ultra thick] (0,0) --(2,0);
\draw [thin] (0,0) -- (0,-2.5);
\draw [thin] (1,0) -- (1,-2.5);
\draw [thin] (2,0) -- (2,-2.5);
\draw[fill=black] (2,-1.5-0.15) circle (0.08cm);
\draw[fill=black] (2,-1-0.15) circle (0.08cm);
\draw[fill=black] (1,-1-0.15) circle (0.08cm);
\draw[fill=black] (2,-0.5-0.15) circle (0.08cm);
\draw[fill=black] (0,-0.5-0.15) circle (0.08cm);
\draw[fill=black] (2,0-0.15) circle (0.08cm);
\draw[fill=black] (1,0-0.15) circle (0.08cm);
\end{tikzpicture}
\qquad
\qquad
{
\Yboxdim10pt
\Yvcentermath0
\yng(5,3,3,2,1,1,1)
}
\]
\caption{The abacus representation of $\lambda = (5,3^2,2,1^3)$ with $m= 7$ beads and $d=3$ runners,
next to the Young diagram representation of $\lambda$.}
\end{figure}

A \emph{semi-standard Young tableau (SSYT)} is a Young diagram whose boxes are
filled with non-negative integers, such that each row is weakly increasing and
each column is strictly increasing. Denote the set of SSYT of shape $\lambda$
with entries in $\{0,\dots, m-1\}$ by $\SSYT(\lambda,m)$. Given $T \in \SSYT(\lambda,m)$,
the \emph{type} of $T$ is the vector $\alpha(T) = (\alpha_0(T),\alpha_1(T),\dots, \alpha_{m-1}(T))$
%given by $\alpha_k(T) = |\{b_{ij} \in T : b_{ij} = k\}|$ for $k = 0,\dots, m-1$. In other words
where $\alpha_k(T)$ counts the number of boxes of $T$ containing the number $k$.

The \emph{Schur polynomial} is defined as
\[
s_{\lambda}(x_0,\dots, x_{m-1}) = \sum_{T \in \SSYT(\lambda,m)} x_0^{\alpha_0(T)}x_1^{\alpha_1(T)} \cdots x_{m-1}^{\alpha_{m-1}(T)}.
\]
The polynomial $s_{\lambda}(x_0,\dots, x_{m-1})$ is symmetric and has several alternative definitions, see \cite{StanleyEnumVol2}. The \emph{principal specialization} of $s_{\lambda}(x_0,\dots, x_{m-1})$ is given by
\[
s_{\lambda}(1,q,q^2,\dots, q^{m-1}) = \sum_{T \in \SSYT(\lambda,m)} q^{|T|},
\] 
where $|T|$ denotes the sum of all entries in $T$.
The following explicit formula is referred to as the \emph{q-hook-content formula} and is due to Stanley (see \cite[Thm 7.21.2]{StanleyEnumVol2}),
\begin{align}
s_{\lambda}(1,q,q^2,\dots, q^{m-1}) = q^{b(\lambda)}\prod_{(i,j) \in \lambda} \frac{[m + c_{i,j}]_q}{[h_{i,j}]_q},
\end{align} \noindent
where $b(\lambda) = \sum_{i=1}^r (i-1)\lambda_i$, $c_{i,j} = j-i$ (the \emph{content}) and $h_{i,j}$ is defined as the number of boxes in $\lambda$ to
the right of $(i, j)$ in row $i$ plus the number of boxes below $(i, j)$ in column $j$ plus $1$ (the \emph{hook length}). In particular
\begin{align} \label{hcontent}
|\SSYT(\lambda,m)| = s_{\lambda}(1^m)= \prod_{(i,j)\in \lambda} \frac{m+c_{i,j}}{h_{i,j}}.
\end{align} \noindent
If $G$ is a group and $V$ a (finite-dimensional) vector space over $\mathbb{C}$, 
then a \emph{representation} of $G$ is a group homomorphism $\rho:G \to \text{GL}(V)$ 
where $\text{GL}(V)$ is the group of invertible linear transformations of $V$. 
A representation $\rho:G \to \text{GL}(V)$ is \emph{irreducible}
if it has no proper subrepresentation $\rho|_W:G \to \text{GL}(W)$, $0 < W < V$ closed under the action of $\{ \rho(g) : g \in G \}$.
The \emph{character} of $G$ on $V$ is a function $\chi:G \to \mathbb{C}$ defined by $\chi(g) = \text{tr}(\rho(g))$.
Note that characters are invariant under conjugation by $G$.
A character $\chi$ is said to be \emph{irreducible} if the underlying representation is irreducible.
If $G = \mathfrak{S}_m$, then the irreducible characters $\chi^{\lambda}$ of $\mathfrak{S}_m$
are indexed by partitions $\lambda$ of weight $m$ and may be computed
combinatorially (on each conjugacy class of type $\alpha$ in $\mathfrak{S}_m$)
using the \emph{Murnaghan--Nakayama rule} \cite[Thm 7.17.3]{StanleyEnumVol2}
\begin{align} \label{murnaka}
\chi_{\alpha}^{\lambda} = \sum_{T \in \text{BST}(\lambda, \alpha)} (-1)^{\text{ht}(T)},
\end{align} \noindent
where the sum runs over all border strip tableaux $\text{BST}(\lambda, \alpha)$
of shape $\lambda$ and type $\alpha$ and $\text{ht}(T)$ is the sum of all heights
of the border strips in $T$. In particular this implies $\chi^{\lambda}$ takes integer values.

The following theorem provides an expression for the root of unity 
evaluation of the principal specialization $s_{\lambda}(1,q,\dots, q^{m-1})$.

\begin{theorem}[Reiner--Stanton--White \cite{Reiner2004}] \label{rootunityspec}
Let $d|m$ and $\omega_d$ be a primitive $d^{th}$ root of unity.
Then $s_{\lambda}(1,\omega_d, \dots, \omega_d^{m-1})$ is zero unless the $d$-core of $\lambda$ is empty, in which case
\[
\schurS_{\lambda}(1,\omega_d,\omega_d^2,\dotsc,\omega_d^{m-1}) 
= sgn(\chi_{d^{|\lambda|/d}}^{\lambda}) \prod_{i=0}^{d-1} \schurS_{\lambda^{(i)}}(1^{m/d}),
\]
where $\chi^{\lambda}$ is the irreducible character of the
symmetric group $\mathfrak{S}_{|\lambda|}$ indexed by $\lambda$. 
\end{theorem}
%Multiplying $\lambda$ with $n$, and letting $m=n$, we get 
\begin{lemma} \label{lem:nRootUnitySpec}
Suppose $\omega_d$ is a primitive $d^{th}$ root of unity with $d|m$, then
\[
\schurS_{n\lambda}(1,\omega_d,\omega_d^2,\dotsc,\omega_d^{m-1}) = \prod_{i=0}^{d-1} \schurS_{(n\lambda)^{(i)}}(1^{m/d}) \in \mathbb{N}.
\]
\end{lemma}
\begin{proof}
If $d$ does not divide $n|\lambda|$, 
then $\schurS_{n\lambda}(1,\omega_d,\omega_d^2,\dotsc,\omega_d^{m-1}) = 0$ by \cref{rootunityspec}, so there is nothing to prove. 
Thus we may assume $d$ divides $n|\lambda|$.
By \cref{rootunityspec} we only need to verify that $\chi_{d^{n|\lambda|/d}}^{n\lambda} \geq 0$. 
A result by White \cite[Cor. 10]{White1983} (see also \cite[Thm. 3.3]{Pak2000Ribbon}),
implies that the Murnaghan--Nakayama rule \eqref{murnaka} is cancellation-free in this instance.
Furthermore, it is clear that there is a border-strip tableau of shape $n\lambda$ with border-strips of size $d$
with positive sign.
For example, take all strips to be horizontal --- this is possible since $d|n$.
\end{proof}

We are now ready to state our conjecture.
\begin{conjecture} \label{conj:schurpolyconj}
Let $n, m \in \setN$ and let $\lambda$ be a partition.
Then the triple 
\[
(\SSYT(n\lambda,m), C_n, \schurS_{n \lambda}(1,q,q^2,\dotsc,q^{m-1}) )
\] 
exhibits a CSP for some $C_n$ acting on $\SSYT(n\lambda,m)$.
\end{conjecture} \noindent
We believe that a natural action is realized by some type of
promotion on semi-standard Young tableaux similar to \cite{Rhoades2010}.
In the case $\lambda=(1)$ we have
\[
 \schurS_{n\lambda}(1,q,q^2,\dotsc,q^m) = \qbinom{n+m-1}{n}_q
\]
and this polynomial exhibits a cyclic sieving phenomenon under $C_n$, see \cite{Reiner2004}.

We have verified \cref{conj:schurpolyconj} using \cref{thm:cspCond}
for all partitions $\lambda$ such that $|\lambda| \leq 6$, all $m\leq 6$ and all $n\leq 12$.

Below we prove the conjecture for certain rectangular shapes $\lambda$.

\begin{lemma}\label{lem:rectQuotient}
The $n$-quotient of the rectangular shape $(na)^{nb+r}$ with $0 \leq r <n$ is given by 
\[
 [\underbrace{a^{b},a^b,\dotsc,a^b}_{\text{$n-r$ times}},\underbrace{a^{b+1},a^{b+1},\dotsc,a^{b+1}}_{\text{$r$ times}}].
\]
%Furthermore, the $d$-quotient of $(na)^b$ with $d|n$
%is a series of rectangular shapes, all with parts of size $(n/d)a$
%and either with $\lfloor b/d \rfloor$ or $\lceil b/d \rceil$ parts.
\end{lemma}
\begin{proof}
The abacus representation of $\lambda = (na)^{nb+r}$ with $m = nb+r$ beads and $d = n$ runners is given via
\[
na + (nb+r) -i = s+nt,
\] 
for $i = 1, \dotsc, nb+r$ where $0 \leq s \leq n-1$, see \cref{fig:abacus2}.
Thus we see that each of the $n$ runners have no bead in the first $a$ rows. 
Since all parts of $\lambda$ are the same, we also note that the $nb+r$ beads 
are distributed evenly from right to left on the $n$ runners with no vacant positions 
in between the beads on each runner. Thus there are $b$ beads on the first $n-r$ runners 
and $b+1$ beads on the last $r$ runners. Moreover each bead have exactly $a$ vacant positions 
above it on its runner, so the $n$-quotient is given as in the lemma.   
\end{proof}

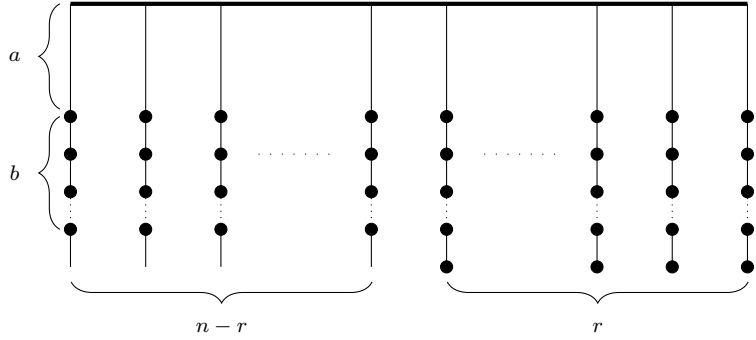
\begin{figure}
\begin{tikzpicture}
\draw [ultra thick] (0,0) --(9,0);
\draw [thin] (0,0) -- (0,-2.5);
\draw [dotted] (0,-2.5) -- (0,-3);
\draw [thin] (0,-3) -- (0,-3.5);

\draw [thin] (1,0) -- (1,-2.5);
\draw [dotted] (1,-2.5) -- (1,-3);
\draw [thin] (1,-3) -- (1,-3.5);

\draw [thin] (2,0) -- (2,-2.5);
\draw [dotted] (2,-2.5) -- (2,-3);
\draw [thin] (2,-3) -- (2,-3.5);

\draw[loosely dotted] (2.5,-2) -- (3.5,-2);

\draw [thin] (4,0) -- (4,-2.5);
\draw [dotted] (4,-2.5) -- (4,-3);
\draw [thin] (4,-3) -- (4,-3.5);

\draw [thin] (5,0) -- (5,-2.5);
\draw [dotted] (5,-2.5) -- (5,-3);
\draw [thin] (5,-3) -- (5,-3.5);

\draw[loosely dotted] (5.5,-2) -- (6.5,-2);

\draw [thin] (7,0) -- (7,-2.5);
\draw [dotted] (7,-2.5) -- (7,-3);
\draw [thin] (7,-3) -- (7,-3.5);

\draw [thin] (8,0) -- (8,-2.5);
\draw [dotted] (8,-2.5) -- (8,-3);
\draw [thin] (8,-3) -- (8,-3.5);

\draw [thin] (9,0) -- (9,-2.5);
\draw [dotted] (9,-2.5) -- (9,-3);
\draw [thin] (9,-3) -- (9,-3.5);

\draw[fill=black] (0,-1.5) circle (0.08cm);
\draw[fill=black] (0,-2) circle (0.08cm);
\draw[fill=black] (0,-2.5) circle (0.08cm);
\draw[fill=black] (0,-3) circle (0.08cm);
\draw[fill=black] (1,-1.5) circle (0.08cm);
\draw[fill=black] (1,-2) circle (0.08cm);
\draw[fill=black] (1,-2.5) circle (0.08cm);
\draw[fill=black] (1,-3) circle (0.08cm);
\draw[fill=black] (2,-1.5) circle (0.08cm);
\draw[fill=black] (2,-2) circle (0.08cm);
\draw[fill=black] (2,-2.5) circle (0.08cm);
\draw[fill=black] (2,-3) circle (0.08cm);
\draw[fill=black] (4,-1.5) circle (0.08cm);
\draw[fill=black] (4,-2) circle (0.08cm);
\draw[fill=black] (4,-2.5) circle (0.08cm);
\draw[fill=black] (4,-3) circle (0.08cm);
\draw[fill=black] (5,-1.5) circle (0.08cm);
\draw[fill=black] (5,-2) circle (0.08cm);
\draw[fill=black] (5,-2.5) circle (0.08cm);
\draw[fill=black] (5,-3) circle (0.08cm);
\draw[fill=black] (5,-3.5) circle (0.08cm);
\draw[fill=black] (7,-1.5) circle (0.08cm);
\draw[fill=black] (7,-2) circle (0.08cm);
\draw[fill=black] (7,-2.5) circle (0.08cm);
\draw[fill=black] (7,-3) circle (0.08cm);
\draw[fill=black] (7,-3.5) circle (0.08cm);
\draw[fill=black] (8,-1.5) circle (0.08cm);
\draw[fill=black] (8,-2) circle (0.08cm);
\draw[fill=black] (8,-2.5) circle (0.08cm);
\draw[fill=black] (8,-3) circle (0.08cm);
\draw[fill=black] (8,-3.5) circle (0.08cm);
\draw[fill=black] (9,-1.5) circle (0.08cm);
\draw[fill=black] (9,-2) circle (0.08cm);
\draw[fill=black] (9,-2.5) circle (0.08cm);
\draw[fill=black] (9,-3) circle (0.08cm);
\draw[fill=black] (9,-3.5) circle (0.08cm);
\draw[decorate,decoration={brace,amplitude=8pt},xshift=-4pt,yshift=0pt]
(0,-1.4) -- (0,0) node [black,midway,xshift=-0.6cm] 
{\footnotesize $a$};
\draw[decorate,decoration={brace,amplitude=8pt},xshift=-4pt,yshift=0pt]
(0,-3) -- (0,-1.5) node [black,midway,xshift=-0.6cm] 
{\footnotesize $b$};
\draw[decorate,decoration={brace,amplitude=8pt},xshift=-4pt,yshift=0pt]
(4.15,-3.7) -- (0.15,-3.7) node [black,midway,yshift=-0.6cm] 
{\footnotesize $n-r$};
\draw[decorate,decoration={brace,amplitude=8pt},xshift=-4pt,yshift=0pt]
(9.15,-3.7) -- (5.15,-3.7) node [black,midway,yshift=-0.6cm] 
{\footnotesize $r$};
\end{tikzpicture}
\caption{The abacus representation of $\lambda = (na)^{nb+r}$ with $m = nb+r$ beads and $d=n$ runners.} \label{fig:abacus2}
\end{figure}

\begin{lemma}\label{lem:rectSchurEval}
We have
 \[
\schurS_{(a^b)}(1^m) = \prod_{j=0}^{a-1} \binom{m+j}{b}\binom{b+j}{b}^{-1} 
 \]
\end{lemma}
\begin{proof}
By the hook-content formula \eqref{hcontent} we have
\begin{align*}
s_{(a^b)}(1^m) = \prod_{(i,j) \in (a^b)} \frac{m+j -i}{(a-j)+ (b-i)+1},
\end{align*} \noindent
which after rearrangement equals
\begin{align*}
\prod_{j=0}^{a-1} \prod_{i=0}^{b-1} \frac{m+j-i}{b+j-i} =
%\prod_{j=0}^{a-1} \frac{(m+j)^{\underline{b}}}{(b+j)^{\underline{b}}} =
\prod_{j=0}^{a-1} \frac{(m+j)!}{(m-b+j)!} \frac{j!}{(b+j)!} = \prod_{j=0}^{a-1} \binom{m+j}{b}\binom{b+j}{b}^{-1}.
\end{align*}
\end{proof} \noindent
\begin{theorem} \label{thm:rectSchur}
Let $n,m,a,b \in \mathbb{N}$ with $b < m$ and $n|b,m$. If $\lambda = (a^b)$, then the triple 
\[
(\SSYT(n\lambda,m), C_n, s_{n\lambda}(1,q,q^2,\dots, q^{m-1}) )
\] 
exhibits a CSP for some $C_n$ acting on $\SSYT(\lambda,m)$.
\end{theorem}
\begin{proof}
By \cref{lem:nRootUnitySpec} it follows that $\schurS_{n\lambda}(1,\omega_n^j,\omega_n^{2j},\dotsc,\omega_n^{(m-1)j}) \in \mathbb{N}$ for all $j = 1, \dots, n$. By \cref{thm:cspCond} it therefore remains to show that for all $k|n$,
\begin{align}\label{eq:mainRectSchurIneq}
\sum_{j|k} \mu(k/j) \schurS_{n\lambda}(1,\omega_n^j,\omega_n^{2j},\dotsc,\omega_n^{(m-1)j}) \geq 0.
\end{align} \noindent
Note that $\omega_n^j$ is a $(n/j)^{th}$ root of unity.
By \cref{lem:nRootUnitySpec} and \cref{lem:rectQuotient} the left hand side of \eqref{eq:mainRectSchurIneq} rewrites as
\begin{align}
\sum_{j|k} \mu(k/j) \prod_{i=0}^{n/j-1} \underbrace{ \schurS_{(ja)^{bj/n}}(1^{mj/n}) }_{\text{independent of $i$}} &=  
\sum_{j|k} \mu(k/j) \left( \schurS_{(ja)^{bj/n}}(1^{mj/n}) \right)^{n/j}.
\end{align}
Using \cref{lem:rectSchurEval}, this equals
\begin{align}
 \sum_{j|k} \mu(k/j) \left( \prod_{i=0}^{ja-1} \binom{mj/n+i}{bj/n}\binom{bj/n+i}{bj/n}^{-1}   \right)^{n/j},
\end{align} \noindent
which is greater or equal to
\begin{align} \label{eq:firstMobiusEstimate}
\left(   \prod_{i=0}^{ka-1} \binom{mk/n+i}{bk/n}\binom{bk/n+i}{bk/n}^{-1}   \right)^{\frac{n}{k}} \!\!\!\!\!\! - \sum_{\substack{j|k \\ j < k}} \left(   \prod_{i=0}^{ja-1} \binom{mj/n+i}{bj/n}\binom{bj/n+i}{bj/n}^{-1}   \right)^{\frac{n}{j}}.
\end{align} \noindent 
By \cref{lem:binomineq} and the fact that the number of divisors of $k$, excluding $k$,
is bounded above by $2\sqrt{k}-1$ we get that \eqref{eq:firstMobiusEstimate} is greater than or equal to
\begin{align} \label{eq:finalMobiusEstimate}
\left( \prod_{i=k'a}^{ka-1} \frac{\binom{mk/n+i}{bk/n}^{n/k}}{\binom{bk/n+i}{bk/n}^{n/k}} - (2\sqrt{k}-1) \right ) \left(\prod_{i=0}^{k'a-1}  \frac{\binom{mk/n+i}{bk/n}^{n/k}}{\binom{bk/n+i}{bk/n}^{n/k}} - \prod_{i=0}^{k'a-1} \frac{\binom{mk'/n+i}{bk'/n}^{n/k'}}{\binom{bk'/n+i}{bk'/n}^{n/k'}}  \right ),
\end{align} \noindent
where $k' = \lfloor k/2\rfloor$.
The remaining steps needed are given in the appendix \cref{sec:appendix},
where it is shown that the left factor in \eqref{eq:finalMobiusEstimate} is non-negative by \cref{lem:prodBinomLwrBnd} and the right factor is non-negative by \cref{lem:binomineq} for all $k|n$.
This concludes the proof of the theorem.
\end{proof}

\section{The CSP cone} \label{sec:cspCone}
In the following sections we offer a geometric perspective on the cyclic sieving phenomenon by associating a polyhedral cone that captures joint information about the cyclic action and statistics on the object $X$. The cone has the property that all cyclic sieving phenomena with a polynomial generated by a choice of statistic (modulo $n$) on the set $X$ corresponds to a lattice point in the cone.

As presented in the introduction, the polynomial $f(q)$ is often given by some natural statistic $\stat : X \to \mathbb{N}$ on $X$. Define
\[
f_{\stat}(q) \coloneqq \sum_{x \in X} q^{\stat(x)}.
\]
Moreover for each $n \in \mathbb{N}$, define $\stat_n : X \to \mathbb{Z}_n$ by 
\[
\stat_n(x) \coloneqq \stat(x) \Mod{n}.
\]
More than understanding the individual components of the CSP triple $(X,C_n,f_{\tau}(q))$, 
one is also interested in the behaviour and distribution of the statistic $\tau$ with respect to the cyclic action.
Given an action of $C_n$ on $X$ and a statistic $\stat:X \to \mathbb{N}$, 
we can associate a $n \times n$ matrix $A_{(X,C_n,\stat)} = (a_{ij})$ 
which keeps track of the coefficients of the generating function
\begin{align*}
\sum_{x \in X} q^{\stat_n(x)}t^{o(x)} \coloneqq \sum_{i=0}^{n-1} \sum_{j=1}^{n} a_{ij}q^{i}t^{j},
\end{align*}
where $o(x) \coloneqq \min\{ j \in [n] : \sigma_n^j \cdot x = x \}$ denotes the order
of $x\in X$ under $C_n$. We remark that the rows of $A_{(X,C_n,\stat)}$ are indexed from $0$ to $n-1$.

We can now restate CSP as follows:
\begin{proposition}\label{prop:integerCSPMatrix}
Suppose $X$ is a finite set on which $C_n$ acts and let $\stat:X\to\mathbb{N}$ be a statistic. Then
the triple $(X,C_n,f_{\stat}(q))$ exhibits CSP if and only if $A_{(X,C_n,\stat)} = (a_{ij})$ 
satisfies the condition that for each $1 \leq k \leq n$,
\begin{equation} \label{eq:CSPCond}
\sum_{\substack{0 \leq i < n \\ 1\leq j \leq n}} a_{ij} \nthroot_n^{ki} = \sum_{0\leq i < n} \sum_{j|k} a_{ij}.
\end{equation}
where $\nthroot_n$ is a primitive $n$th root of unity.
\end{proposition}
\begin{proof}
For each $1 \leq k \leq n$ we have that
\begin{align*}
X^{\sigma_n^k} &= \bigcup_{i=0}^{n-1} \{ x \in X : \stat_n(x) = i, \thickspace \sigma_n^k \cdot x = x \} \\
               &= \bigcup_{i=0}^{n-1} \bigcup_{j|k} \{ x \in X : \stat_n(x) = i, \thickspace o(x) = j \}.
\end{align*}
Hence $(X,C_n,f_{\stat}(q))$ exhibits CSP if and only if for each $1 \leq k \leq n$,
\begin{align}
\sum_{\substack{0\leq i < n \\ 1\leq j \leq n}} a_{ij} \nthroot_n^{ki} = f_{\stat_n}(\nthroot_n^k) =  |X^{\sigma_n^k}| = \sum_{0\leq i < n} \sum_{j|k} a_{ij}.
\end{align} 
\end{proof}
\noindent
This motivates the following definition.
\begin{definition}
A $n\times n$-matrix $A = (a_{ij}) \in \setR_{\geq 0}^{n \times n}$ is called a \emph{CSP-matrix} if it fulfills 
the conditions in \cref{eq:CSPCond}. 
Let $\CSP(n)$ denote the set of all $n \times n$ $\CSP$-matrices 
and $\CSP_{\mathbb{Z}}(n) \coloneqq \CSP(n) \cap \mathbb{Z}^{n \times n}$ the set of integer CSP-matrices.
\end{definition} \noindent
\begin{example}
Consider all binary words of length $6$, with group action being shift by $1$
and $\stat$ being the the major index statistic.
Then
% n = 6;
% words = Tuples[{0, 1}, n];
% mat = CSPMatrix[words, MajorIndex, RotateRight[#, 1] &, n];
% mat // MatrixForm
% FullSimplify@CheckCSPMatrix[mat]
\[
\begin{pmatrix}
 2 & 1 & 0 & 0 & 0 & 11 \\
 0 & 0 & 2 & 0 & 0 & 7 \\
 0 & 0 & 0 & 0 & 0 & 11 \\
 0 & 1 & 2 & 0 & 0 & 7 \\
 0 & 0 & 0 & 0 & 0 & 11 \\
 0 & 0 & 2 & 0 & 0 & 7 \\
\end{pmatrix} 
\]
is the corresponding CSP matrix. The entry in the upper left hand corner correspond
to the two binary words $000000$ and $111111$.
These have major index $0$ and are fixed under a single shift. The words corresponding to the second column are $010101$ and $101010$.  These have major index $6 \equiv 0 \Mod{6}$ and $9 \equiv 3 \Mod{6}$ respectively and are fixed under two consecutive shifts etc.
\end{example}
By linearity of the CSP-condition \eqref{eq:CSPCond},
it follows that for all $A,B \in \CSP(n)$ we have $sA+tB \in \CSP(n)$ for any $s,t\geq 0$. Hence $\CSP(n)$ forms a real convex cone. In fact by Theorem \ref{thm:nhypdesc} in Section \ref{sec:geomCSP} we have the following corollary.
\begin{corollary} \label{cor:convRatPolyhedral}
The set $\CSP(n)$ forms a real convex rational polyhedral cone.
\end{corollary}
\section{General properties of the CSP cone} \label{sec:generalProps}
Since $\CSP(n)$ is a rational cone by \cref{cor:convRatPolyhedral}, its extreme rays are spanned by integer matrices. Every element in $\CSP(n)$ is therefore a conic combination of elements in $\CSP_{\mathbb{Z}}(n)$. In particular, properties of $\CSP_{\mathbb{Z}}(n)$ closed under conic combinations can be lifted to $\CSP(n)$.

A priori an integer lattice point $A \in \CSP_{\mathbb{Z}}(n)$ need not be realizable by a cyclic sieving phenomenon with CSP-matrix $A$. However thanks to Lemma \ref{lem:aprioriprops} we shall see that this property does indeed hold.

\begin{lemma} \label{lem:aprioriprops}
Let $A = (a_{ij})\in \CSP_{\mathbb{Z}}(n)$. Then there exists a CSP-triple $(X,C_n,\tau)$ with $A_{(X,C_n,\tau)} = A$. 
\end{lemma}
\begin{proof}
According to (\ref{eq:CSPCond}), the polynomial $f(q) = \sum_{i=0}^{n-1} r_i q^{i}$ where $r_i = \sum_{j =1}^n a_{ij}$ for $i = 0,\dots, n-1$ defines a polynomial such that
$f(\omega_n^k) = \sum_{j|k} S_j \in \mathbb{N}$ for $k=1,\dots, n$,
where $S_j = \sum_{i=0}^{n-1} a_{ij}$ for $j = 1, \dots, n$.
By Möbius inversion as in Theorem \ref{thm:cspCond} we have
$S_k = \sum_{j|k} \mu(k/j) f(\omega_n^j)$.
Hence by Lemma \ref{lem:mobiusDiv}, $k|S_k$.
Therefore a CSP-instance having CSP-matrix $A$ can be realized through any triple $(X,C_n,\tau)$ with $C_n$ acting  in an ad-hoc manner on a set $X$ with $\sum_{i,j} a_{ij}$ elements divided into $S_k/k$ orbits of size $k$ for each $k|n$ where $\tau:X \to \mathbb{N}$ is any statistic distributed according to $A$.  
\end{proof} \noindent

Let $\{ E_{ij} : 0 \leq i < n, 1 \leq j \leq n \}$ denote the standard basis of $\mathbb{R}^{n \times n}$. 
\begin{definition}
Call a matrix $\delta_a(\mathbf{u}, \mathbf{v}) \in \mathbb{R}^{n \times n}$ a \emph{swap} if
\[
\delta_a(\mathbf{u}, \mathbf{v}) \coloneqq a(E_{u_1u_2} + E_{v_1v_2} - E_{v_1u_2} - E_{u_1v_2}),
\]
where $a \in \mathbb{R}$.
\end{definition}\noindent
\begin{lemma} \label{lem:dispInvariance}
Let $A \in \CSP(n)$ and suppose $\delta_a(\mathbf{u}, \mathbf{v}) + A \in \mathbb{R}_{\geq 0}^{n \times n}$. Then $\delta_a(\mathbf{u}, \mathbf{v}) + A \in \CSP(n)$.
\end{lemma}
\begin{proof}
Since adding $\delta_a(\mathbf{u}, \mathbf{v})$ does not alter column nor row-sums we have that the CSP-condition (\ref{eq:CSPCond}) remains intact. Hence $\delta_a(\mathbf{u}, \mathbf{v}) + A \in \CSP(n)$. 
\end{proof} \noindent
The next lemma follows by repeated applications of \cref{lem:dispInvariance}.
\begin{lemma} \label{lem:permAct}
Let $A = (a_{ij}) \in \CSP(n)$. Suppose $i$ and $i'$ are two row indices such that $\sum_{j=1}^n a_{ij} = \sum_{j=1}^n a_{i'j}$. If $A'$ is the matrix obtained from $A$ by interchanging rows $i$ and $i'$, then $A' \in \CSP(n)$. 
\end{lemma}
\begin{remark}
The corresponding statement of Lemma \ref{lem:permAct} also holds for the column indices instead of row indices.
\end{remark}
\begin{proposition}
Let $n \in \mathbb{N}$ and suppose $i$ and $i'$ are row indices such that $\gcd(n,i) = \gcd(n,i')$. If $A \in \CSP(n)$, then $A' \in \CSP(n)$ where $A'$ is obtained from $A$ by interchanging rows $i$ and $i'$.
\end{proposition}
\begin{proof}
Let $A \in \CSP_{\mathbb{Z}}(n)$. Then the polynomial $f(q) = \sum_{i=0}^{n-1} c_iq^i \in \mathbb{N}[q]$, where $c_i = \sum_{j=1}^n a_{ij}$, satisfies $f(\omega_n^j) \in \mathbb{N}$ for all $j = 1,\dots, n$. By Lemma \ref{lem:modBasis} it follows that $c_{i \Mod{n}} = c_{\text{gcd}(n,i)}$ for all $i = 1, \dots, n$. Hence $A' \in \CSP_{\mathbb{Z}}(n)$ by Lemma \ref{lem:permAct}. Moreover from above, row $i$ and $i'$ clearly have the same row sum in $sA+tB$ for any $A,B \in \CSP_{\mathbb{Z}}(n)$ and $s,t \geq 0$. Hence the property can be lifted to all matrices in $\CSP(n)$. 
\end{proof} \noindent

\section{The universal CSP cone}  \label{sec:univCone}
Let $W_\alpha$ be the set of words with content $\alpha$, that is, $\alpha_i$
is the number of occurrences of the letter $i$ in the words, and let $n$
be the length of the words. Then $C_n$ acts on such words by cyclic shift.
In \cite{Ahlbach2017}, the authors construct a statistic, $\flex(\cdot)$,
which is equidistributed modulo $n$ with major index on $W_\alpha$.
Furthermore, $\flex$ has the property that for every orbit $\orbit$, 
the triple $(\orbit,C_n,\flex)$
exhibits the cyclic sieving phenomenon.
They show that $\flex$ is universal in the following sense:
\begin{definition}
A cyclic sieving phenomena $(X,C_n,\stat)$ is called \emph{universal} if $(\orbit, C_n, \stat)$ exhibits the cyclic sieving phenomenon for every orbit $C_n$-orbit $\orbit$ of $X$.
This is shown in \cite{Ahlbach2017} to be equivalent with the property that
for every $C_n$-orbit $\orbit \subseteq X$ with length $k$,
the sets
\[
\{\stat_n(x): x \in \orbit\} \text{  and } \left\{0, \frac{n}{k}, \frac{2n}{k}, \dotsc, \frac{(k-1)n}{k} \right\}
\]
coincide.
In other words, the statistic $\stat$ is ``evenly distributed'' on each $C_n$-orbit modulo $n$.
We also refer to $\stat$ as being a universal statistic (with respect to $X$ and $C_n$).
\end{definition}

Clearly a universal statistic is uniquely determined modulo $n$ by the orbit structure of $X$ under $C_n$ (up to a choice of total order on the orbits).
We remark that most cyclic sieving phenomena in the literature are \emph{not}
universal. We shall see below how a non-universal
statistic can be turned into a universal one without changing the generating polynomial.
\begin{definition} \label{def:univCSPMat}
A matrix $A = (a_{ij}) \in \CSP(n)$ is called \emph{universal} if there are constants $K_1, \dots, K_n \in \mathbb{R}_{\geq 0}$ such that
\[
a_{ij} = \begin{cases} K_j, & \text{ if } i \equiv 0 \Mod{\frac{n}{j}}, \\ 
0, & \text{otherwise.}  
\end{cases}
\]
for all $1 \leq i,j\leq n$.
Let $\widetilde{\CSP}(n)$ denote the subset of all universal CSP-matrices. Moreover if $\svec = (S_1, \dots, S_n) \in \mathbb{N}^n$ is a sequence such that $j|S_j$ for $j = 1, \dots, n$ and $S_j = 0$ for $j \not | n$, then we let $U(\svec) \in \widetilde{\CSP}(n)$ denote the unique universal CSP-matrix with column sums given by $S_1,\dots, S_n$. 
\end{definition} \noindent
\begin{remark}
Note that $\widetilde{\CSP}(n)$ forms a subcone of $\CSP(n)$ and that the lattice points $\widetilde{\CSP}_{\mathbb{Z}}(n)$ are realized by universal cyclic sieving phenomena.
\end{remark} \noindent
Every CSP-matrix can be linearly projected onto a universal CSP-matrix. Indeed the map
\begin{align*}
P:\CSP(n) &\to \widetilde{\CSP}(n) \\
  a_{ij} &\mapsto \begin{cases} \frac{1}{j} \sum_{i=1}^n a_{ij}, & \text{ if } i \equiv 0 \Mod{\frac{n}{j}}, \\ 0, & \text{ otherwise, } \end{cases}
\end{align*} \noindent
is clearly linear in each entry with $P^2 = P$. By Proposition \ref{lem:aprioriprops} the projection $P$ restricts to a map $P:\CSP_{\mathbb{Z}}(n) \to \widetilde{\CSP}_{\mathbb{Z}}(n)$.

If $A \in \CSP_{\mathbb{Z}}(n)$, then $\delta_1(\mathbf{u},\mathbf{v}) + A$ corresponds to swapping statistic between two elements belonging to orbits of different size. 

Next we show that every CSP matrix $A \in \CSP_{\mathbb{Z}}(n)$ can be obtained from a universal CSP-matrix with the same column sums via a sequence of such swaps while keeping inside $\CSP_{\mathbb{Z}}(n)$. We prove this fact by showing a slightly more general result over the class of non-negative integer matrices with matching row and column sums.
\begin{proposition} \label{prop:swap}
Let $A=(a_{ij})$ and $B = (b_{ij})$ be integer $n \times n$ matrices with non-negative entries having matching row and column sums i.e. $\sum_{i=1}^n a_{ij_0} = \sum_{i=1}^n b_{ij_0}$ and $\sum_{j=1}^n a_{i_0j} = \sum_{j=1}^n b_{i_0j}$ for $1 \leq i_0,j_0 \leq n$. Then there exists swaps $\delta_1(\uvec_r,\vvec_r)$ for $r = 1, \dots, t$ such that 
\begin{equation}
     A = B + \sum_{r=1}^t \delta_1(\mathbf{u}_r, \mathbf{v}_r).
\end{equation} \noindent
Moreover the swaps $\delta_1(\mathbf{u}_r, \mathbf{v}_r)$ can be chosen such that $B + \sum_{r=1}^{t_0} \delta_1(\mathbf{u}_r, \mathbf{v}_r)$ has non-negative entries for all $1 \leq t_0 \leq t$.
\end{proposition}
\begin{proof}
Define $\Delta(A)$ to be the quantity 
\[
\Delta(A) \coloneqq ||A-B||
\] 
where $||A|| = \sum_{i,j} |a_{ij}|$.
We say that an entry $a_{ij}$ is in \emph{deficit} if $a_{i j} < b_{ij}$ and in \emph{surplus} if $a_{i j} > b_{ij}$.
We argue by induction on $\Delta(A)$. If $\Delta(A) = 0$, then clearly $A = B$ since $A$ and $B$ both have non-negative entries. Suppose $\Delta(A) > 0$. Then there exists indices $i$ and $j$ such that $a_{ij}-b_{ij} \neq 0$. If $a_{ij}$ is in surplus, then there must exists some row index $i'$ such that $a_{i'j}$ is in deficit, otherwise the sum of column $j$ in $A$ is strictly greater than sum of column $j$ in $B$ which leads to a contradiction. Therefore we may assume $a_{ij}$ is in deficit. Since $a_{ij}$ is in deficit there exists $j' \neq j$ such that $a_{ij'}$ is in surplus, otherwise the sum of row $i$ in $B$ is strictly greater than the sum of row $i$ in $A$. Similarly, there exists a row index $i' \neq i$ such that $a_{i'j}$ is in surplus. It follows that
\[
A' \coloneqq A - \delta_1((i,j'),(i',j))
\] 
has non-negative entries by construction with row and column sums matching that of $A$ (and hence that of $B$). Moreover 
\[ 
\Delta(A')  = \begin{cases} \Delta(A)-4, & \text{ if } a_{i'j'} \text{ is in deficit}, \\ \Delta(A)-2, & \text{ otherwise}. \end{cases}.
\]  
Hence by induction
\begin{align*}
A &=  A' + \delta_1((i,j'),(i',j)) \\ &= B +  \sum_{r=1}^t \delta(\mathbf{u}_r, \mathbf{v}_r) + \delta_1((i,j'),(i',j)).
\end{align*} \noindent
\end{proof}
\begin{corollary}
Let $A = (a_{ij}) \in \CSP_{\mathbb{Z}}(n)$. Write $S_j = \sum_{i=0}^{n-1} a_{ij}$ for the column sums of $A$ for $j=1, \dots, n$ and set $\svec = (S_1,\dots, S_n)$. Then there exists swaps $\delta_1(\mathbf{u}_r, \mathbf{v}_r)$ for $r=1, \dots,t$ such that 
\begin{equation}
     A = U(\svec) + \sum_{r=1}^t \delta_1(\mathbf{u}_r, \mathbf{v}_r).
\end{equation} \noindent
Moreover $U(\svec) + \sum_{r=1}^{t_0} \delta_1(\mathbf{u}_r, \mathbf{v}_r) \in \CSP_{\mathbb{Z}}(n)$ for all $1 \leq t_0 \leq t$.
\end{corollary}
\begin{proof}
Let $R_i = \sum_{j=1}^n a_{ij}$ denote the row sums of $A$ for $i = 0,1,\dots, n-1$. Note that the row sums of $A$ are determined uniquely by the column sums of $A$ via
\[ 
R_i = \sum_{j: \frac{n}{j}|i} \frac{1}{j}S_j,  \hspace{0.3cm} \text{ for } i = 0,\dots, n-1,
\]
since both sides count the number of orbits whose stabilizer-order divides $i$ in the corresponding CSP-instance, according to \eqref{eq:equivDefRSW} and Remark \ref{rem:numElements}. Since $A$ and $U(\svec)$ have the same column sums they must therefore have the same row sums. The corollary now follows from \cref{prop:swap} and \cref{lem:dispInvariance}.
\end{proof}

\begin{remark}
\cref{prop:swap} shows that every $A \in \CSP_{\mathbb{Z}}(n)$ can be uniquely expressed as $U(\svec) + B$ where $B=(b_{ij}) \in \mathbb{Z}^{n \times n}$ is a matrix with zero row and column-sums and non-negative values in all entries $b_{k \ell}$ unless $(k,\ell) = (\frac{ni}{j},j)$ where $0 \leq i < j$ and $j|n$.
\end{remark}
\begin{construction}
If $(C_m,X,f(q))$ and $(C_n,Y,g(q))$ are two CSP-triples, then we can construct a new CSP-triple of the form $(C_{mn}, X \times Y, h(q))$ where $h(q)$ is a polynomial of degree less than $mn$ which may be expressed as certain convolution of $f$ and $g$. 

Let $(x,y) \in X \times Y$ and suppose $o(x) = i$, $o(y)= j$ with respect to the actions of $C_m$ on $X$ and $C_n$ on $Y$ respectively.
Let $C_{mn}$ act on $(x,y)$ via
\[
\sigma_{mn}^{is+t} \cdot (x,y) \coloneqq (\sigma_m^{t} \cdot x, \sigma_n^{s} \cdot y) 
\]
where $0 \leq t < i$ and $s \in \mathbb{Z}$. Note that $(x,y)$ has order $ij$ under the above action. By Remark \ref{rem:numElements}, the number of elements of order $i$ and $j$ with respect to the actions of $C_m$ on $X$ and $C_n$ on $Y$ are given respectively by 
\begin{align*}
S_i = \sum_{\ell | i} \mu(\ell/i)f(\omega_m^{\ell}), \hspace{0.2cm} T_j = \sum_{\ell | j} \mu(\ell/j)g(\omega_n^{\ell}).
\end{align*} \noindent
Therefore the action of $C_{mn}$ on $X \times Y$ has 
\[
\sum_{ij = k} S_i T_j,
\]
elements of order $k$. By \eqref{eq:equivDefRSW} the coefficients $c_r$ of the unique polynomial $h(q) = \sum_{r=0}^{mn-1} c_rq^r$ (mod $q^{mn}-1$) complementing the action of $C_{mn}$ on $X \times Y$ to a CSP is given by the number of orbits whose stabilizer-order divides $r$, that is,
\[
c_r = \sum_{k: \frac{mn}{k} | r }\sum_{ij = k} \frac{1}{k} S_iT_j.
\] 
\end{construction}
The above construction gives rise to a natural product on universal CSP-matrices. Given a vector $\svec = (S_d)$, we define its \defin{number-theoretical series} as
the formal power-series
\begin{align}
NS(\svec) \coloneqq \sum_{1 \leq d} S_d x^{e_1}_{p_1}\dotsc x^{e_\ell}_{p_\ell}
\end{align}
where $d = p_1^{e_1} \dotsc p_{\ell}^{e_\ell}$ is the prime factorization of $d$.

Given two vectors $\svec$ and $\tvec$ of length $m$ and $n$, respectively,
define the vector $\svec \boxtimes \tvec$ of length $mn$ via the identity
\[
 NS(\svec \boxtimes \tvec) = NS(\svec)\cdot NS(\tvec).
\]
In other words, coordinate $k$ in $\svec \boxtimes \tvec$ is given by $\sum S_i T_j$,
where the sum ranges over all natural numbers $i$, $j$ such that $ij=k$.
Note that $\boxtimes$ is symmetric and transitive, and $|\svec \boxtimes \tvec| = |\svec|\cdot|\tvec|$
where $|\cdot|$ denotes the sum of the entries.
\medskip

\begin{proposition}
Let $U(\svec) \in \widetilde{\CSP}(m)$ and $U(\tvec) \in \widetilde{\CSP}(n)$. Then 
\[
U(\svec) \boxtimes U(\tvec) \coloneqq U(\svec \boxtimes \tvec) \in \widetilde{\CSP}(mn).
\]
\end{proposition}
\begin{proof}
We have that $i|S_i$ and $j|T_j$ for $i = 1, \dots, m$, $j= 1, \dots, n$ and $S_i,T_j = 0$ for $i \not | m, \thickspace j \not | n$. It follows that
\[
(\svec \boxtimes \tvec)_k = \sum_{\substack{ij = k \\ i|m, j|n}} S_iT_j,
\]
with $k|(\svec \boxtimes \tvec)_k$ for $k = 1, \dots, mn$ and $(\svec \boxtimes \tvec)_k = 0$ if $k \not | mn$.
\end{proof}
\section{Geometry of the CSP cone}  \label{sec:geomCSP}
The below theorem provides the half-space description of $\CSP(n)$, showing that it is indeed a rational convex polyhedral cone. 
\begin{theorem} \label{thm:nhypdesc}
Let $n \in \setN \setminus \{0\}$ and $A= (a_{ij}) \in \mathbb{R}^{n \times n}$.
Let the divisors of $n$ be given by
\[
1 = c_1 < c_2 < \cdots < c_d = n.
\] 
Let
\[
H_k(\xvec) \coloneqq \sum_{i=0}^{n-1} \sum_{j=2}^d \alpha_{ijk} x_{ij} \in \mathbb{Z}[\mathbf{x}],
\]
where
\begin{align*}
\alpha_{ijk} \coloneqq \begin{cases} -n+ \frac{n}{c_j}, & \text{ if }  i = k \text{ and } k \equiv 0\Mod{\frac{n}{c_j}}, \\
-n & \text{ if } i = k \text{ and } k \not \equiv 0\Mod{\frac{n}{c_j}}, \\ 
\frac{n}{c_j}, & \text{ if } i \neq k \text{ and } k \equiv 0\Mod{\frac{n}{c_j}}, \\ 
0 & \text{ if } i \neq k \text{ and } k \not \equiv 0\Mod{\frac{n}{c_j}}. 
\end{cases}   
\end{align*} \noindent
Then $A$ is a CSP matrix if and only if
\[
A = (\avec_1|\avec_2|\cdots |\avec_n),
\]
where 
\begin{align*}
\avec_1 &= (x_{01}, H_1(\xvec), \dots, H_{n-1}(\xvec))^t, \\ \avec_c &= \begin{cases} (nx_{0c}, nx_{1c}, \dots, nx_{(n-1)c})^t, & \text{ if } c | n, \\  \mathbf{0}, & \text{ otherwise, } \end{cases}
\end{align*} \noindent
for $c = 2,\dots, n$ with $H_k(\xvec) \geq 0$ and $x_{ij} \geq 0$ for all $i,j,k$.
\end{theorem}
\begin{proof}
For $\zvec \in \setC^{n-1}$, let
\[
V(\zvec) \coloneqq 
\begin{pmatrix}
z_1     & z_1^2     & \hdots & z_1^{n-1} \\
z_2     & z_2^2     & \hdots & z_2^{n-1} \\
\vdots  & \vdots    &        & \vdots    \\
z_{n-1} & z_{n-1}^2 & \hdots & z_{n-1}^{n-1}
\end{pmatrix}.
\]
Let $\nthrootvec \coloneqq (\omega_n,\omega_n^2, \dots, \omega_n^{n-1})$ and set
\[
B_j \coloneqq (\mathbf{1}^t|V(\nthrootvec)) - J_{c_j}
\] 
for $j = 1, \dots, d$ where $\mathbf{1} \coloneqq (1, \dots, 1) \in \setR^{n-1}$ and
\[
J_{c_j}(k,\ell) \coloneqq \begin{cases} 1 & \text{ if } c_j|k, \\ 0 & \text{ otherwise }\end{cases}
\]
for $1 \leq k \leq n-1$ and $1 \leq \ell \leq n$.
Consider the matrix
\[
B \coloneqq \left [B_1| B_2| \cdots |B_d \right].
\]
Then $A = (a_{ij}) \in \mathbb{R}_{\geq 0}^{n \times n}$ satisfies \eqref{eq:CSPCond} if and only if
\begin{equation} \label{eq:csphom}
B\avec = \mathbf{0},
\end{equation}
where $\avec = (\avec_1| \cdots |\avec_d)^t$
and $\avec_j = (a_{1c_j}, \dots, a_{nc_j})$ for $j = 1, \dots, d$. 
Note that the defining CSP-equations \eqref{eq:CSPCond} immediately give that $a_{ij} = 0$ for all $1 \leq i \leq n$ and $j \nmid n$.
We claim that the real solutions to \eqref{eq:csphom} are of the form
\begin{equation} \label{eq:CSPSol}
\avec_1 = \begin{pmatrix} x_{01} \\ H_1(\xvec) \\ \vdots \\ H_{n-1}(\xvec)  \end{pmatrix}, \hspace{0.5cm} 
\avec_j = \begin{pmatrix} nx_{0j} \\ nx_{1j} \\ \vdots \\ nx_{(n-1)j} \end{pmatrix}
\end{equation} \noindent
where $x_{01},x_{ij} \in \setR$ for $0 \leq i \leq n-1$, $2 \leq j \leq d$ and 
\[
H_k(\xvec) = \sum_{i=0}^{n-1} \sum_{j=2}^d \alpha_{ijk} x_{ij}
\]
for some $\alpha_{ijk} \in \setZ$, $k = 1, \dots, n-1$. 
Since $B$ has full rank $n-1$, the solutions \eqref{eq:CSPSol} make up the whole null space of $B$ for dimensional reasons. Thus we only need to concern ourselves with the existence of solutions of the form \eqref{eq:CSPSol}.

Given \eqref{eq:csphom} and supposing \eqref{eq:CSPSol} we thus require
\begin{equation}
(V(\nthrootvec)-J_1) \boldsymbol{\alpha}^{(ij)} = \uvec^{(ij)},
\end{equation} \noindent
for $i = 0, \dots, n-1$ and $j = 2, \dots, d$
where
\begin{align*}
\boldsymbol{\alpha}^{(ij)} \coloneqq \begin{pmatrix}
\alpha_{ij1} \\ \alpha_{ij2} \\ \vdots \\ \alpha_{ij(n-1)}
\end{pmatrix}, \hspace{0.5cm}  \uvec^{(ij)} &\coloneqq \begin{pmatrix}
u^{(ij)}_1 \\ u^{(ij)}_2 \\ \vdots \\ u^{(ij)}_{n-1}
\end{pmatrix}, \hspace{0.5cm}
u^{(ij)}_{k} \coloneqq 
\begin{cases} -n\omega_n^{ik} + n, & \text{ if } c_j | k, \\ 
-n\omega_n^{ik}, & \text{ otherwise }   
\end{cases}. \hspace{0.2cm}
\end{align*} \noindent
Note that 
\[
(V(\nthrootvec)-J_1)^{-1} = \frac{1}{n} V(\overline{\nthrootvec}).
\]
Therefore
\[
\boldsymbol{\alpha}^{(ij)} = \frac{1}{n}V(\overline{\nthrootvec}) \uvec^{(ij)},
\]
which gives
\begin{align*}
\alpha_{ijk} &= \sum_{\ell=1}^{n-1}\frac{\overline{\omega}_n^{k \ell}}{n} (-n\omega_n^{i \ell }) + \sum_{\substack{\ell = 1\\ c_j | \ell}}^{n-1} \frac{\overline{\omega}_n^{k \ell}}{n}n   \\ &= - \sum_{\ell = 0}^{n-1} (\omega_n^{(i-k)})^{\ell} + \sum_{s = 0}^{\frac{n}{c_j}-1} ((\omega_n^{c_j})^k)^s  \\ &= \begin{cases} -n+ \frac{n}{c_j}, & \text{ if }  i = k \text{ and } k \equiv 0\Mod{\frac{n}{c_j}}, \\ -n & \text{ if } i = k \text{ and } k \not \equiv 0\Mod{\frac{n}{c_j}}, \\ \frac{n}{c_j}, & \text{ if } i \neq k \text{ and } k \equiv 0\Mod{\frac{n}{c_j}}, \\ 0 & \text{ if } i \neq k \text{ and } k \not \equiv 0\Mod{\frac{n}{c_j}}. \end{cases}
\end{align*} \noindent
Hence the theorem follows.
\end{proof} \noindent
The following corollary follows immediately from Theorem \ref{thm:nhypdesc}. 
\begin{corollary}
Let $n \in \mathbb{N} \setminus \{0\}$ and $d \coloneqq |\{ c \in \mathbb{N} : c|n  \}|$ denote the number of divisors of $n$.
Then $\CSP(n)$ has dimension $n(d-1)+1$.
\end{corollary} \noindent
Recall that a polyhedral cone is given by
$P = \{ \xvec \in \mathbb{R}^n: A\xvec \geq \bvec \}$ for some $n \times n$ matrix $A$.
A non-zero element $\xvec$ of a polyhedral cone $P$ is called an \emph{extreme ray} if there are $d-1$ linearly independent constraints that are active at $\xvec$ (\emph{i.e.} hold with equality at $\xvec$). If $\xvec$ is an extreme ray, then $\lambda \xvec$ is also an extreme ray for $\lambda > 0$.
Two extreme rays that are positive multiples of each other are called \emph{equivalent}.
Equivalent extreme rays correspond to the same $d-1$ active constraints. Extreme rays can also be defined as points in $\xvec \in P$ that cannot be expressed as a convex combination of two points in the interior of $P$.

Below we give an explicit description of a subset of the extreme rays of $\CSP(n)$. This subset includes all extreme rays of the universal CSP-cone $\widetilde{\CSP}(n)$ (see Corollary \ref{cor:univExtremeRays}). When $n=p$ for some prime number $p$, then we get all the extreme rays (see Corollary \ref{cor:prays}).
\begin{theorem} \label{thm:extremeRays}
Let $n \in \mathbb{N}\setminus \{0\}$ and suppose 
\[
1 = c_1 < c_2 < \cdots < c_{d-1}< c_d = n
\] 
are the divisors of $n$. Let $\ell_0 \in [d]$. Then $\rvec = (r_{ij}) \in \mathbb{R}^{n \times n}$ is an extreme ray of $\CSP(n)$ if 
\[
r_{ij} = \begin{cases} 
1, & \text{ if } (i,j) = (0,c_{\ell_0}), \\ 
\frac{1}{c_{\ell_0}-|I|}, & \text{ if } i \in  I \text{ and } j = c_{\ell_0}, \\ 0, & \text{ otherwise } \end{cases}
\]
for $0 \leq i \leq n-1$ and $1 \leq j \leq n$
where $I \subseteq \{ t \frac{n}{c_{\ell_0}} \in \mathbb{N} : 1 \leq t < c_{\ell_0} \}$. In particular the number of extreme rays of $\CSP(n)$ is at least
\[
\frac{1}{2}\sum_{\ell=1}^{d} 2^{c_\ell}.
\]  
\end{theorem}
\begin{proof}
By Theorem \ref{thm:nhypdesc}, $\CSP(n)$ is isomorphic to the polyhedral cone
\[
\{ \xvec \in \mathbb{R}^{n(r-1)+1} : n\xvec \geq 0 \text{ and } H_k(\xvec) \geq 0 \text{ for all } k = 1, \dots, n-1 \}.
\]
Let $\rvec = (r_{ij}) \in \mathbb{R}^{n \times n}$ be an extremal ray of $\CSP(n)$ such that $r_{ij} = 0$ if $j \neq c_{\ell_0}$.
Note that the defining inequalities of $\CSP(n)$ imply in particular that $r_{ij} \geq 0$ for all $0 \leq i < n$ and $1 \leq j \leq d$.

Suppose first that $r_{0c_{\ell_0}}=0$. Let $k \in [n-1]$ be such that $r_{kc_{\ell_0}} \geq r_{ic_{\ell_0}}$ for all $i \in [n-1]$.
Suppose for a contradiction that $r_{kc_{\ell_0}} > 0$.
The maximality of $r_{kc_{\ell_0}}$ implies $\frac{r_{ic_{\ell_0}}}{r_{kc_{\ell_0}}} \leq 1$.
The defining inequalities of the polyhedral cone $\CSP(n)$ gives $-nr_{ic_{\ell_0}} \geq 0$ for $i \not \equiv 0 \Mod{\frac{n}{c_{\ell_0}}}$, which implies that $r_{ic_{\ell_0}} = 0$ for $i \not \equiv 0 \Mod{\frac{n}{c_{\ell_0}}}$.
Thus we may assume $k \equiv 0 \Mod{\frac{n}{c_{\ell_0}}}$. Now, $H_k(\xvec) \geq 0$ gives
\[
0 \leq -n + \frac{n}{c_{\ell_0}} + \sum_{\substack{i \in [n-1] \setminus k \\ \frac{n}{c_{\ell_0}} | i}} \frac{r_{ic_{\ell_0}}}{r_{kc_{\ell_0}}} \leq -n + \frac{n}{c_{\ell_0}} + c_{\ell_0}-2,
\] 
which holds if and only if
$c_{\ell_0} \leq \frac{n+2}{2} - \Delta$ or $c_{\ell_0} \geq \frac{n+2}{2} + \Delta$ where $\Delta = \left ( \left ( \frac{n+2}{2} \right )^2 -n \right )^{1/2}$. Since $\frac{n+2}{2} - \Delta < 1$ and $\frac{n+2}{2} + \Delta > n$ for $n > 0$ whereas $1 \leq c_{j_0} \leq n$ this gives a contradiction.

Hence we may assume $r_{0c_{\ell_0}} > 0$. Let 
\[
M_{\ell} \coloneqq \left \{ t \frac{n}{c_\ell} \in \mathbb{N} : 1 \leq t < c_{\ell} \right \}.
\]
Suppose $I \subseteq M_{c_{\ell_0}}$ such
that $r_{ic_{\ell_0}} > 0$ for $i \in I$ and $r_{ic_{\ell_0}} = 0$ for $i \in M_{c_{\ell_0}} \setminus I$.
Since $\rvec$ is an extreme ray there are by definition $n(d-1)$ linearly
independent constraints active at $\rvec$. Since $r_{ij} = 0$ for $j \neq c_{\ell_0}$
and $r_{ic_{\ell_0}} = 0$ for $i \in [n-1]\setminus I$ there are $n(d-2) + 1 + (n-1)- |I|$ active constraints covered.
Note that we have $(-n + \frac{n}{c_{\ell_0}})r_{kc_{\ell_0}} + \sum_{i \neq k} \frac{n}{c_{\ell_0}}r_{ic_{\ell_0}} > 0$
for $k \not \in I$ and $nr_{kc_{\ell_0}} > 0$ for $k \in M_{c_{\ell_0}}\setminus I$.
Hence the remaining $|I|$ inequalities must be active at $\rvec$ which gives
\begin{equation} \label{eq:nconstr}
\left (-n + \frac{n}{c_{\ell_0}} \right )r_{kc_{\ell_0}} + \sum_{i \neq k} \frac{n}{c_{\ell_0}}r_{ic_{\ell_0}} = 0
\end{equation} \noindent
for $k \in I$. If $I = \emptyset$, then the only non-zero entry of $\rvec$ is $r_{0c_{\ell_0}}$. 
Suppose $I \neq \emptyset$. Summing the equations (\ref{eq:nconstr}) and dividing by $\frac{n}{c_{\ell_0}}r_{0c_{\ell_0}}$, we get
\[
0 = \frac{c_{\ell_0}}{nr_{0c_{\ell_0}}} \sum_{i \in I} \left ( \left ( -n + \frac{n}{c_{\ell_0}} \right ) r_{ic_{\ell_0}} + \sum_{k \neq i} \frac{n}{c_{\ell_0}} r_{k c_{\ell_0}} \right ) = (-c_{\ell_0}+ |I|)\sum_{i \in I} \frac{r_{ic_{\ell_0}}}{r_{0c_{\ell_0}}} + |I|.
\]
Hence we get the average ratio
\begin{equation} \label{eq:avgratio}
\frac{1}{|I|} \sum_{i \in I} \frac{r_{ic_{\ell_0}}}{r_{0c_{\ell_0}}} = \frac{1}{c_{\ell_0}- |I|}
\end{equation}
Suppose
\[
\frac{r_{kc_{\ell_0}}}{r_{0c_{\ell_0}}} > \frac{1}{c_{\ell_0}- |I|}
\]
for some $k \in I$.
Then by dividing \eqref{eq:nconstr} with $\frac{n}{c_{\ell_0}}r_{0c_{\ell_0}}$ and using \eqref{eq:avgratio} we have
\begin{align*}
0 &= \left ( -c_{\ell_0} + 1 \right )\frac{r_{kc_{\ell_0}}}{r_{0c_{\ell_0}}} + 1 + \sum_{i\in I} \frac{r_{ic_{\ell_0}}}{r_{0c_{\ell_0}}} - \frac{r_{k c_{\ell_0}}}{r_{0c_{\ell_0}}} \\ &= -c_{\ell_0} \frac{r_{kc_{\ell_0}}}{r_{0c_{\ell_0}}} + 1 + \frac{|I|}{c_{\ell_0} - |I|} \\ &< \frac{-c_{\ell_0}}{c_{\ell_0} - |I|} + 1 + \frac{|I|}{c_{\ell_0} - |I|} = 0,
\end{align*} \noindent
which gives a contradiction. Hence by \eqref{eq:avgratio} we have that 
\[
r_{ic_{\ell_0}} = \frac{r_{0c_{\ell_0}}}{c_{\ell_0}- |I|}
\] 
for all $i\in I$ proving the theorem.

\end{proof}
\begin{corollary} \label{cor:univExtremeRays}
Let $n \in \mathbb{N}\setminus \{0\}$ and suppose $1 = c_1 < c_2 < \cdots <c_{d-1} < c_d = n$ are the divisors of $n$.
Let $M_{\ell} = \{ t \frac{n}{c_{\ell}} : 0 \leq t < c_{\ell} \}$ and
define $\rvec^{(\ell)} = (r_{ij}^{(\ell)}) \in \mathbb{R}^{n \times n}$ by
\[
r_{ij}^{(\ell)} = \begin{cases} 
1, & \text{ if } i \in M_{\ell} \text{ and } j = c_{\ell}, \\ 
 0, & \text{ otherwise, } \end{cases}
\]
for $1 \leq \ell \leq d$.
Then the extreme rays of $\widetilde{\CSP}(n)$ are given by $\{ \rvec^{(\ell)} : 1 \leq \ell \leq d \}$.
\end{corollary}
\begin{proof}
By Theorem \ref{thm:extremeRays} the set $\{ \rvec^{(\ell)} : 1 \leq \ell \leq d \}$ are indeed extreme rays and they clearly generate all universal CSP matrices (cf. Definition \ref{def:univCSPMat}).
\end{proof}
\begin{corollary} \label{cor:prays}
Let $p \in \mathbb{N}$ be a prime number. Then the extreme rays of $\CSP(p)$ are given by $E_{01} \in \mathbb{R}^{p \times p}$ and $\rvec = (r_{ij}) \in \mathbb{R}^{p \times p}$ such that
\[
r_{ij} = \begin{cases} 
1, & \text{ if } (i,j) = (0,p), \\ 
\frac{1}{p-|I|}, & \text{ if } i \in  I \text{ and } j = p, \\ 0, & \text{ otherwise, } \end{cases}
\]
where $I \subseteq \{1, \dots, p-1 \}$. 
In particular the number of extreme rays of $\CSP(p)$ is given by $2^{p-1}+1$.
\end{corollary} \noindent
By adding a size restriction on the set $X$ we can also talk about a natural family of polytopes associated with cyclic sieving phenomena.
\begin{definition}
Let $m \in \mathbb{N}$. The $m^{\text{th}}$ \emph{CSP-polytope} is the convex rational polytope defined by
\[
\CSP(n,m) \coloneqq \left \{ A \in \CSP(n) : ||A|| = m  \right \}.
\]
Let $\CSP_{\mathbb{Z}}(n,m) \coloneqq \CSP(n,m) \cap \mathbb{Z}^{n \times n}$ denote the set of integer lattice points in $\CSP(n,m)$.
\end{definition} \noindent
Once again, in the case where $n=p$ for some prime number $p \in \mathbb{N}$ we are able to make explicit computations. In the following two propositions we compute the vertices and the number of integer lattice points of $\CSP(n,m)$.
\begin{proposition}
Let $p \in \mathbb{N}$ be a prime number and $m \in \mathbb{N}$. 
Then the vertices of $\CSP(p,m)$ are given by $mE_{01} \in \mathbb{R}^{p \times p}$
and $\vvec = (v_{ij}) \in \mathbb{R}^{p \times p}$ such that
\[
v_{ij} = \begin{cases} 
C, & \text{ if } (i,j) = (0,p), \\ 
\frac{C}{p-|I|}, & \text{ if } i \in  I \text{ and } j = p, \\ 0, & \text{ otherwise, } \end{cases}
\]
where $I \subseteq \{2, \dots, p \}$ and 
\[
C = \frac{m}{2p-1+ \frac{(p-1)|I|}{p- |I|}}.
\]
In particular the number of vertices of $\CSP(p,m)$ is given by $2^{p-1}+1$.
\end{proposition}
\begin{proof}
Suppose $\vvec = (v_{ij}) \in \mathbb{R}^{p \times p}$ is a vertex of $\CSP(p,m)$.

If $v_{0p} = 0$, then arguing as in the first part of the proof of Theorem \ref{thm:extremeRays} gives that $v_{0p} = v_{1p} = \cdots = v_{p-1p} =0$. Therefore $v_{ij} = 0$, unless $j = 1$ by Lemma \ref{lem:aprioriprops}. 
The additional constraint $||\vvec|| = m$ thus gives
\begin{equation} 
m = \sum_{\substack{0\leq i < p\\ 1 \leq j \leq p}} v_{ij} = \sum_{j=1}^p v_{0j} =  x_{01} + \sum_{k=1}^{p-1} H_k(\xvec),
\end{equation} \noindent
which is the same as
\begin{equation} \label{eq:ptopeconstr}
x_{01} + (2p-1)x_{0p} + (p-1)x_{1p} + \cdots + (p-1)x_{p-1p} = m.
\end{equation} \noindent
Since $x_{ip} = v_{ip}$ for $i = 0,1 \dots,p-1$ we get that $v_{01} = x_{01} = m$, so that $\vvec = mE_{01}$.

Therefore suppose $v_{0p} > 0$. Moreover suppose $I \subseteq \{ 1, \dots, p-1 \}$ 
such that $v_{ip} > 0$ for $i \in I$ and $v_{ip} = 0$ for $i \in \{1, \dots, p-1 \} \setminus I$. 
Since $\vvec$ is a vertex, there are by definition $p+1$ linearly independent 
constraints active at $\vvec$. Since $p$ of these constraints arise from the 
polyhedral description of $\CSP(p)$ in \cref{thm:nhypdesc} it follows, as in \cref{thm:prays}, that
\[
v_{ip} = \begin{cases} C, & \text{ if } (i,j) = (0,p),\\
\frac{C}{p-|I|}, & \text{ if } i \in I, \\ 
0, & \text{ if } i \in \{1, \dots, p-1 \} \setminus I,
\end{cases}
\]
for some $C > 0$.
The remaining active constraint is \cref{eq:ptopeconstr}. 
Inserting the above into \cref{eq:ptopeconstr} and solving for $C$ yields
\[
C = \frac{m}{2p-1+ \frac{(p-1)|I|}{p- |I|}},
\] 
from which the proposition follows.
\end{proof}
\begin{proposition}
Let $p,m \in \mathbb{N}$ where $p$ is a prime number. The number of lattice points in $\CSP(p,m)$ is given by
\[
|\CSP_{\mathbb{Z}}(p,m)| = \sum_{j=0}^m \sum_{r \in \left [ \frac{2j}{2p-1}, \frac{j}{p-1} \right ] \cap \mathbb{Z}} C \left (r(2p-1) - 2j, p-1, \lfloor r - j/p \rfloor \right),
\]
where 
\[
C(n,k,w) = \sum_{j=0}^k (-1)^j \binom{k}{j} \binom{n - jw-1}{k-1}.
\]
\end{proposition}
\begin{proof}
Let $x = x_{0p}$, $y = x_{1p} + \cdots + x_{p-1p}$ and $z = x_{01}$. 
According to the constraint $||A|| = m$ we seek non-negative integer solutions to
\[
(2p-1)x + (p-1)y + z = m, 
\]
(cf. Equation \eqref{eq:ptopeconstr}) satisfying $H_k(\xvec) \geq 0$. We therefore consider the Diophantine equations
\[
(2p-1)x + (p-1)y = j,
\]
for $j = 0, \dots, m$ which have the non-negative integer solutions
\begin{align*}
x = j - r(p-1) \text{ and } y = -2j + r(2p-1), 
\end{align*} \noindent
for $r \in \left [ \frac{2j}{2p-1}, \frac{j}{p-1} \right ] \cap \mathbb{Z}$. 
The constraints $H_k(\xvec) \geq 0$ for $k = 1, \dots, p-1$ give
\[
x - (p-1)x_{kp} + (y-x_{kp}) \geq 0,
\]
which implies
\[
x_{kp} \leq r - \frac{j}{p}
\]
for $k = 1, \dots, p-1$. Hence the lattice points in $\CSP(p,m)$ are 
in one-to-one correspondence with weak compositions 
of $y = -2j + r(2p-1)$ into $p-1$ parts of size at most $\lfloor r - j/p \rfloor$. 
By \cite{Abramson1976} the number of such compositions are given by $C(r(2p-1)-2j,p-1,\lfloor r-j/p \rfloor)$.   
\end{proof}

\section{Appendix}   \label{sec:appendix}
In this appendix we prove inequalities needed for the estimations in Theorem \ref{thm:rectSchur}. The first inequality below gives a sufficient condition for a Riemann sum to be monotonically increasing. A slightly weaker result appears in \cite[Theorem 3A]{Bennet2000}.

\begin{proposition}\label{eq:functionInequality}
Let $f(x)$ be a \emph{decreasing} convex\footnote{For all $a$, $b$ we have $f(\frac{a+b}{2})\leq\frac{f(a)+f(b)}{2}$,
or equivalently for twice differentiable functions, $f''\geq 0$.} 
function on $\mathbb{R}_{\geq 0}$, let $p$ be a positive integer and $r \geq 0$.
Then
\begin{align}\label{eq:samplingInequality}
 \frac{1}{p} \sum_{\ell=1}^{p} f\left(\frac{\ell+r}{p} \right) \leq \frac{1}{p+1} \sum_{\ell=1}^{p+1} f\left( \frac{\ell+r}{p+1} \right).
\end{align}
\end{proposition}
\begin{proof}
Let $x_i \coloneqq (i+r)/p$ and $y_i \coloneqq (i+r)/(p+1)$ and note that
\begin{align}\label{eq:weightedAvg}
 x_i  = \left(1-\frac{i}{p}\right) y_i + \frac{i}{p}y_{i+1} + \frac{r}{p(p+1)}.
\end{align}
Since $f$ is decreasing and convex, we have that
\[
 f(x_i) \leq f\left[  \left(1-\frac{i}{p}\right) y_i + \frac{i}{p}y_{i+1} \right] \leq 
   \left(1-\frac{i}{p}\right) f(y_i) + \frac{i}{p}f(y_{i+1})
\]
Now let $a_i \coloneqq f(x_i)$ and $b_i \coloneqq f(y_i)$
and note that the decreasing property implies 
\[
a_i\leq \left(1-\frac{i}{p}\right)b_i + \frac{i}{p} b_{i+1} \leq \left(1-\frac{i}{p+1}\right)b_i + \frac{i}{p+1} b_{i+1}
\text{ for $i=1,\dotsc,p$.}
\]
We add all these inequalities and obtain
\begin{align*}
\sum_{i=1}^{p} a_i  \leq  \frac{1}{p+1} \sum_{i=1}^{p} (p+1 - i) b_i + \frac{1}{p+1}\sum_{i=1}^{p} i b_{i+1}.
\end{align*}
We then have
\begin{align*}
(p+1) (a_1+\dotsb+ a_p)  &\leq  \sum_{i=1}^{p} (p+1 - i) b_i + \sum_{i=2}^{p+1} (i-1) b_{i} \\
         &\leq p \sum_{i=1}^{p}b_i + \sum_{i=1}^{p} (1 - i) b_i + p b_{p+1} + \sum_{i=2}^{p} (i-1) b_{i} \\
         &\leq p (b_1 + \dotsb + b_{p+1}).
\end{align*}
This implies \eqref{eq:samplingInequality}.
\end{proof}

\begin{corollary}
Let $r,s \geq 0$ and $p \in \setN$. Then the expression
\begin{align}\label{eq:basicIncreasing}
g(p) = \frac{1}{p}\sum_{\ell=1}^{p}  \frac{1}{s+(r+\ell)/p}
\end{align}
is increasing with $p$.
\end{corollary}
\begin{proof}
Choosing the decreasing convex function $f(x) = 1/(s+x)$ in \cref{eq:functionInequality} 
together with the given $r$ yields
\begin{align*}
 \frac{1}{p}\sum_{\ell=1}^{p}  \frac{1}{s+(r+\ell)/p}\leq \frac{1}{p+1}\sum_{\ell=1}^{p+1}  \frac{1}{s+(r+\ell)/(p+1)}.
\end{align*}
\end{proof}

\begin{corollary} \label{cor:harmonicInequality}
If $a$, $t$, $i$, $j$ and $k$ are non-negative integers such that $a\leq t$ and $j\leq k$,
then
\begin{align}\label{eq:harmonicInequality}
 \sum_{\ell=0}^{ka-1}  \frac{1}{kt+i-\ell} \geq  \sum_{\ell=0}^{ja-1} \frac{1}{jt+i-\ell}.
\end{align}
\end{corollary}
\begin{proof}
Choosing $p=ak$, $s = (t-a)/a$ and $r = i$  in \eqref{eq:basicIncreasing} gives that
\begin{align*}
f(ak) = \frac{1}{ka} \sum_{\ell=1}^{ak} \frac{1}{s + (i+\ell)/(ak)}
  = \sum_{\ell=1}^{ka} \frac{1}{kt-ka + i +\ell} 
  &= \sum_{\ell=0}^{ka-1} \frac{1}{kt + i -\ell}.
\end{align*}
The fact that $f(ka) \geq f(kj)$ if $k\geq j$ now gives the desired inequality.
\end{proof}
\begin{lemma} \label{lem:binomineq}
If $a$, $b$, $i$, $j$ and $k$ are non-negative integers such that $a\leq b$ and $j\leq k$,
then
\begin{align}\label{eq:binomineq}
\frac{\binom{kb+i}{ka}^{1/k}}{\binom{ka  + i}{ ka }^{1/k}} \geq  \frac{\binom{jb+i}{ ja }^{1/j}}{\binom{ ja + i}{ ja }^{1/j}}. 
\end{align}
\end{lemma}
\begin{proof}
The inequality can be rewritten as $f(b) \geq f(a)$, where
\begin{align}
f(t) \coloneqq \frac{\binom{kt+i}{ka}^j}{ \binom{jt+i}{ ja }^{k}  }.
\end{align}
Thus, it suffices to show that $f(t)$ is increasing.
Computing the derivative and factoring out positive terms reduces to \cref{eq:harmonicInequality}.
\end{proof}
\begin{remark}
In the case where $j|k$, the binomial inequality \eqref{eq:binomineq} admits the following combinatorial interpretation.
A certain organization wants $ka$ members to sit on its executive committee and $ja$ members to sit on the committee for each of its $k/j$ factions. Then the number of possible committee constellations with $kb+i$ candidates for the executive committee and $ja+i$ candidates for each of the factions, is greater than the number of committee constellations with $ka+i$ candidates for the executive committee and $jb+i$ candidates for each faction.  
\end{remark}
\begin{lemma} \label{lem:prodBinomLwrBnd}
If $a,b$ and $k$ are non-negative integers such that $b > a$, then for each $0 \leq i \leq ka$ we have
\[
   \binom{kb+i}{ka} \binom{ka+i}{ka}^{-1} \geq \left ( \frac{b+a}{2a} \right )^{ka}
\]
\end{lemma}
\begin{proof}
If $B>A$, then the function $f(x) = \frac{B+x}{A+x}$ is decreasing as $x$ increases. 
Thus for $0 \leq i \leq ka$ we have
\begin{align*}
\binom{kb+i}{ka}\binom{ka+i}{ka}^{-1} &= \frac{(kb+i)(kb+i-1)\cdots (kb+i-ka+1)}{(ka+i)(ka+i-1)\cdots (i+1)} \\ &\geq \left ( \frac{kb+i}{ka+i} \right )^{ka} \\ &\geq \left ( \frac{b+a}{2a} \right )^{ka}
\end{align*} \noindent
\end{proof} \noindent
\textbf{Acknowledgements.} The authors would like to thank Darij Grinberg, Gleb Nenashev, Mehtaab
Sawhney and Olof Sisak for helpful discussions. The first author is funded by the
Knut and Alice Wallenberg Foundation (2013.03.07).

\bibliographystyle{amsalpha}
\bibliography{thebibliography}

\end{document}